\documentclass{amsart}
\usepackage[style=alphabetic,backend=biber, sorting=nyt]{biblatex}
\bibliography{references}

\usepackage{amssymb}
\usepackage{bbm}
\usepackage{amsmath}
\usepackage{enumitem}
\usepackage{tikz-cd}
\usepackage{faktor}

\usepackage{tikz}
\usetikzlibrary{shapes.geometric}
\tikzset{
v/.style={draw, fill, circle, minimum size=1.5mm, inner sep=0},
b/.style={draw , regular polygon,regular polygon sides=4, minimum size=1.5mm, inner sep=.5mm},
e/.style={very thick},
vs/.style={draw, fill, circle, minimum size=1mm, inner sep=0},
bs/.style={draw,  regular polygon,regular polygon sides=4, minimum size=2mm, inner sep=0mm},
es/.style={thick}
}
\newlength{\nodeheight}
\newlength{\nodewidth}
\usetikzlibrary{arrows,matrix,positioning}

\usepackage{amsthm}
\theoremstyle{plain}
\newtheorem{theorem}{Theorem}[section]
\theoremstyle{definition} \newtheorem{definition}[theorem]{Definition}
\theoremstyle{plain}
\newtheorem{lemma}[theorem]{Lemma}
\theoremstyle{plain} \newtheorem{proposition}[theorem]{Proposition}
\theoremstyle{plain} 
\theoremstyle{plain} \newtheorem{corollary}[theorem]{Corollary}
\theoremstyle{remark} 
\theoremstyle{remark} \newtheorem{remark}[theorem]{Remark}
\theoremstyle{remark} \newtheorem{example}[theorem]{Example}
\theoremstyle{remark} 
\theoremstyle{definition} 

\renewcommand{\t}{\mathbbm{1}}
\newcommand{\rook}{\mathcal{R}}
\newcommand{\RB}{\mathcal{R}\mathrm{Br}}
\newcommand{\Tor}{\mathrm{Tor}}
\newcommand{\Brauer}{\mathrm{Br}}
\newcommand{\TL}{\mathrm{TL}}
\newcommand{\mOne}{splice}
\newcommand{\mTwo}{deletion}
\newcommand{\maxTpt}{canonical}

\begin{document}

\title{Idempotents and homology of diagram algebras}

\author{Guy Boyde}
\address{Mathematical Institute, Utrecht University, Heidelberglaan 8
3584 CS Utrecht, The Netherlands}
\email{g.boyde@uu.nl}

\subjclass[2020]{Primary 20J06, 16E40; Secondary 20B30}
\keywords{Homology of algebras, rook algebra, Temperley-Lieb algebra, Brauer algebra, idempotents}

\begin{abstract} This paper provides a systematization of some recent results in homology of algebras. Our main theorem gives criteria under which the homology of a diagram algebra is isomorphic to the homology of the subalgebra on diagrams having the maximum number of left-to-right connections. From this theorem, we deduce the `invertible-parameter' cases of the Temperley-Lieb and Brauer results of Boyd-Hepworth and Boyd-Hepworth-Patzt. We are also able to give a new proof of Sroka's theorem that the homology of an odd-strand Temperley-Lieb algebra vanishes, as well as an analogous result for Brauer algebras and an interpretation of both results in the even-strand case. Our proofs are relatively elementary: in particular, no auxiliary chain complexes or spectral sequences are required. We briefly discuss the relationship to cellular algebras in the sense of Graham-Lehrer.
\end{abstract}

\maketitle

\section{Introduction}

The study of homology of algebras begins by observing that the homology of a group $G$ is $\Tor_*^{RG}(\t,\t)$, and that this expression still makes perfectly good sense if one replaces $RG$ by any associative algebra $A$ with a choice of trivial module (or augmentation) $\t$ \cite[Section 2.4.4]{Benson}.

Boyd, Hepworth, and Patzt have pioneered this point of view in the context of homological stability, establishing a variety of results for Iwahori-Hecke algebras of type $A$ \cite{Hepworth}, Temperley-Lieb algebras \cite{BH,BHComb},  Brauer algebras \cite{BHP}, and partition algebras \cite{BHP2}. Patzt has also studied the representation stability of these families \cite{Patzt}. Sroka \cite{Sroka} has since improved the known results on Temperley-Lieb algebras when the number of strands is odd, and Moselle \cite{Moselle} has recently shown that Iwahori-Hecke algebras of type $B$ exhibit homological stability.

These algebras all have an interpretation in terms of diagrams. In some cases, the above authors are able to completely compute the homology, rather than just establishing stability. In these `global' computations, the homology of the algebra always turns out to be isomorphic to that of a certain retract, and this retract always has the same description: it is `the subalgebra on diagrams having the maximum number of left-to-right connections' (see Proposition \ref{retract}). This is generally a good thing: this subalgebra is isomorphic to a group algebra, and the group usually turns out to be a familiar one whose homology is known.

The goal of this paper is to provide an explanation of this phenomenon, using our main technical result, Theorem \ref{subalgebraOfBrauer}. We will focus on Brauer and Temperley-Lieb algebras, as studied in \cite{BH,BHP,Sroka}, as well as two classes whose homology does not yet appear to have been studied: the rook and rook-Brauer algebras.

Among the existing global computations, Sroka's proof that the homology of a Temperley-Lieb algebra on an odd number of strands vanishes \cite{Sroka} is distinguished: the other results all hold when the parameter defining the algebra is invertible, but Sroka's theorem holds whenever the number of strands is odd. The setup of Theorem \ref{subalgebraOfBrauer} will allow us to recover Sroka's result (as Theorem \ref{theoremSroka}), prove an analogue for Brauer algebras (Theorem \ref{BrauerSroka}) and give an interpretation of both results in the even-strand case (Theorems \ref{generalisedSroka} and \ref{generalisedBrauerSroka}). These results are our main applications. We will also recover invertible-parameter homology vanishing for Temperley-Lieb algebras (Theorem \ref{TLRecovery}, originally \cite{BH}), and for Brauer algebras (Theorem \ref{BrauerRecovery}, originally \cite{BHP}), and add a few similar results for rook and rook-Brauer algebras. As shorthand, we will refer to these two classes of result as `Sroka-type' and `invertible-parameter' respectively.

The main theorem is quite technical, so we will begin by giving precise statements of these applications, in Subsection \ref{subsectionResults}, deferring the statement of the main theorem to Subsection \ref{subsectionMaintheorem}, and the definitions of the algebras involved to Section \ref{sectionAlgs}.

Since we deal only in global computations, we recover only the substantially easier part of the work of Boyd-Hepworth \cite{BH} and Boyd-Hepworth-Patzt \cite{BHP}: the harder part of their work consists of `true' homological stability results that hold only in a range. As such, one way of thinking of this paper is as an attempt to recover the global part of the existing theory `direct from the algebra', avoiding the need to handcraft an acyclic chain complex, or deploy any spectral sequences. The `inductive' character of the methods introduced by Boyd-Hepworth will still be visible, especially in the proof of the main theorem.

In another direction, the Temperley-Lieb and Brauer algebras are \emph{cellular} in the sense of Graham and Lehrer \cite{GrahamLehrer}, and the structure on which our methods depends is closely related to the cellular structure (but not identical, see Remark \ref{rmk:cellVsLS}). We hope to explore this connection more in future work, but here we confine ourselves to a few comments, directed at readers familiar with the basic language of cellular algebras, in the interest of accessibility and self-containedness. Such a reader may be especially interested in Remark \ref{rmk:cellularApplicability}, where we suggest that the sensitivity of the cellular technology to the ground ring, which is an advantage in representation theory, is likely to be a disadvantage here.

We hope that these features will make this paper useful in bridging the gap between the homological stability community and mathematicians in other areas who are interested in the (rook-)Brauer and Temperley-Lieb algebras.

\subsection{Applications} \label{subsectionResults}

As mentioned earlier, our applications come in two flavours: those where the defining parameter(s) are invertible, and Sroka-type results which depend on the parity of the number of strands. The most interesting results are those of the second sort, so we begin there. Definitions of the algebras appearing here are given in Section \ref{sectionAlgs}.

\subsubsection{Sroka-type results}

Sroka \cite{Sroka} has proven the following, which holds even for non-invertible $\delta$.

\begin{theorem}[\cite{Sroka}] \label{theoremSroka} Let $n$ be odd. Then $\Tor^{\TL_n(\delta)}_q(\t,\t)=0$ for $q>0$. \end{theorem}

We are able to offer one interpretation of what this should mean for even $n$. First, note that the Temperley-Lieb algebra $\TL_n = \TL_n(\delta)$ contains diagrams having no left-to-right connections if and only if $n$ is even, or in other words, if $I_0 \leq \TL_n$ is the two-sided ideal which is free as an $R$-module on diagrams with no left-to-right connections, then $\TL_n \cap I_0 = 0$ if and only if $n$ is odd. We will prove the following theorem, which (therefore) coincides with Sroka's result when $n$ is odd.

\begin{theorem} \label{generalisedSroka} Let $n$ be any positive integer. Then $$\Tor^{\faktor{\TL_n(\delta)}{I_0}}_q(\t,\t)=0$$ for all $q>0$. \end{theorem}

From this point of view, the moral is that the significant point of Sroka's result is not the parity itself, but rather whether diagrams with no connections are permitted. This is also true from the viewpoint of Sroka's paper, and our proof, although quite different to Sroka's, has echoes of some of the same basic phenomena. The ill-behaviour of diagrams with no left-to-right connections is also familiar from the cellular point of view, as for example in the classification of irreducible representations of $\TL_n$ over a field in \cite[Corollary 6.8]{GrahamLehrer}.

Boyd, Hepworth, and Patzt \cite{BHP} show that the homology of the Brauer algebras is stably isomorphic to that of the symmetric groups. They show that their result is sharp for $n=2$, and ask whether it is sharp in general. This motivates the next theorem, which is our main new application.

\begin{theorem} \label{BrauerSroka} Let $n$ be odd. Then $$\Tor^{\Brauer_n(\delta)}_*(\t,\t) \cong \Tor_*^{R \Sigma_n}(\t,\t) =: H_*(\Sigma_n ; R).$$ \end{theorem}

This extends the results of \cite{BHP}, showing that in fact the identification of the homology with that of the symmetric groups holds globally, at least when $n$ is odd, thereby answering their sharpness question in odd parity. As with Sroka's theorem, Theorem \ref{BrauerSroka} follows immediately from a result that holds regardless of the parity of $n$, by using the fact that $\Brauer_n \cap I_0=0$ when $n$ is odd.

\begin{theorem} \label{generalisedBrauerSroka} Let $n$ be any positive integer. Let $I_0 \leq \Brauer_n(\delta)$ be the two-sided ideal that is free as an $R$-module on diagrams with no left-to-right connections. Then $$\Tor^{\faktor{\Brauer_n(\delta)}{I_0}}_*(\t,\t) \cong \Tor_*^{R \Sigma_n}(\t,\t) =: H_*(\Sigma_n ; R).$$ \end{theorem}

\subsubsection{Invertible-parameter results}

For the rook algebras $\rook_n(\varepsilon)$, we will show (Theorem \ref{rookInvertible}) that if $\varepsilon$ is invertible then the homology of the rook algebras is just that of the symmetric groups, which has been completely computed by Nakaoka \cite{Nakaoka}. This is the simplest application, and does not require the full strength of Theorem \ref{subalgebraOfBrauer}; instead, (a special case of) the much simpler Theorem \ref{quotientingPlus} suffices. The reader who wishes to understand the framework of this paper is encouraged to begin with this application.

By the same method, we show that if $\varepsilon$ is invertible then, regardless of the value of $\delta$, the homology of a rook-Brauer algebra coincides with that of its Brauer subalgebra (Theorem \ref{rookBrauerInvertible}). As a consequence (Corollary \ref{funOne}), these rook-Brauer algebras inherit the stability range for the Brauer algebras due to Boyd-Hepworth-Patzt \cite{BHP}. Our first application of Theorem \ref{subalgebraOfBrauer} is then to recover another theorem from that paper as Theorem \ref{BrauerRecovery}, namely that when $\delta$ is invertible, the homology of the Brauer algebras coincides with that of the symmetric groups.

For the Temperley-Lieb algebras, we recover a result of Boyd-Hepworth as Theorem \ref{TLRecovery}: when $\delta$ is invertible, the homology vanishes.

\subsection{The main theorem} \label{subsectionMaintheorem}

We now assemble the terminology necessary for the statement of the main theorem (Theorem \ref{subalgebraOfBrauer}).

We will use `diagram algebra' (informally) to mean a `nice enough' subalgebra of the rook-Brauer algebra, which we now define. Our ground ring $R$ will always be commutative with unit, and, at least in the introduction, pictures are drawn for $n=5$.

Martin and Mazorchuk \cite{MartinMazorchuk} (who call them \emph{partial Brauer algebras}) state that the first appearance of the rook-Brauer algebras is in the physics literature, namely that they are implicit in the paper \cite{GrimmWarnaar} of Grimm and Warnaar, but it is not clear (at least to the present author) exactly what is meant by this. They are a variant of the much older Brauer algebras, which we will meet later (Definition \ref{defBrauer}). The analogous semigroup has also been studied, at least as far back as Mazorchuk's paper \cite{MazorchukPre}. Our notation will be closest to that of Halverson and delMas \cite{HalversondelMas}.

\begin{definition} A \emph{rook-Brauer} $n$-diagram is a graph consisting of two columns of $n$ nodes (vertices) such that each node is connected to at most one other node by an edge.

Let $R$ be a commutative ring with unit, let $n$ be a natural number, and let $\delta, \varepsilon \in R$. An element of the \emph{rook-Brauer algebra} $\rook \Brauer_n = \rook \Brauer_n(\delta, \varepsilon)$ is a formal $R$-linear combination of rook-Brauer $n$-diagrams. The multiplication is the bilinear extension of a multiplication rule defined on diagrams. For diagrams $x$ and $y$, the product $x \cdot y$ (sometimes $xy$) will be another diagram multiplied by a prefactor of the form $\delta^a \varepsilon^b$. It is defined via the following procedure:
\begin{itemize}
    \item Concatenate $x$ and $y$ by identifying the right hand nodes of $x$ with the left hand nodes of $y$.
    \item Forgetting the vertices in the middle of this concatenated diagram gives an object which differs from a rook-Brauer diagram only in the possible presence of isolated components (sequences of edges) in the middle. These are dealt with as follows, depending on their topological type:
    \begin{itemize}
    \item Each loop should be replaced with a factor of $\delta$.
    \item Each contractible component should be replaced with a factor of $\varepsilon$.
\end{itemize}
\end{itemize}
The result is the (definition of the) product $x \cdot y$.
\end{definition}

We illustrate this multiplication rule with an example.

\begin{example} In $\rook \Brauer_5$, if 
\begin{center}
\[
x =
\begin{tikzpicture}[x=1.5cm,y=-.5cm,baseline=-1.05cm]
\node[v] (a1) at (0,0) {};
\node[v] (a2) at (0,1) {};
\node[v] (a3) at (0,2) {};
\node[v] (a4) at (0,3) {};
\node[v] (a5) at (0,4) {};

\node[v] (b1) at (1,0) {};
\node[v] (b2) at (1,1) {};
\node[v] (b3) at (1,2) {};
\node[v] (b4) at (1,3) {};
\node[v] (b5) at (1,4) {};

\draw[e] (a2) to[out=0, in=0] (a3);

\draw[e] (a1) to[out=0, in=180] (b5);
\draw[e] (a4) to[out=0, in=180] (b2);

\draw[e] (b1) to[out=180, in=180] (b4);

\end{tikzpicture}
\quad
\textrm{ and } y = 
\quad
\begin{tikzpicture}[x=1.5cm,y=-.5cm,baseline=-1.05cm]

\node[v] (b1) at (0,0) {};
\node[v] (b2) at (0,1) {};
\node[v] (b3) at (0,2) {};
\node[v] (b4) at (0,3) {};
\node[v] (b5) at (0,4) {};

\node[v] (c1) at (1,0) {};
\node[v] (c2) at (1,1) {};
\node[v] (c3) at (1,2) {};
\node[v] (c4) at (1,3) {};
\node[v] (c5) at (1,4) {};

\draw[e] (b5) to[out=0, in=180] (c4);

\draw[e] (b1) to[out=0, in=0] (b4);

\draw[e] (c1) to[out=180, in=180] (c3);

\end{tikzpicture}
\]
\end{center}

then the product $x \cdot y$ is
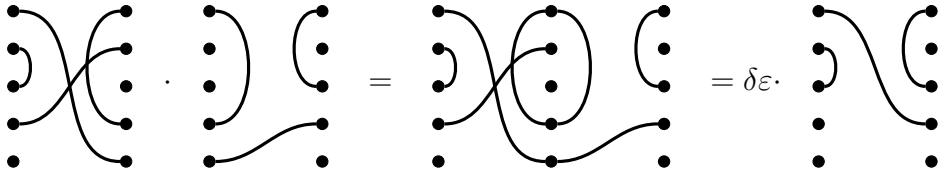
\begin{center}
\[
\begin{tikzpicture}[x=1.5cm,y=-.5cm,baseline=-1.05cm]
\node[v] (a1) at (0,0) {};
\node[v] (a2) at (0,1) {};
\node[v] (a3) at (0,2) {};
\node[v] (a4) at (0,3) {};
\node[v] (a5) at (0,4) {};

\node[v] (b1) at (1,0) {};
\node[v] (b2) at (1,1) {};
\node[v] (b3) at (1,2) {};
\node[v] (b4) at (1,3) {};
\node[v] (b5) at (1,4) {};

\draw[e] (a2) to[out=0, in=0] (a3);

\draw[e] (a1) to[out=0, in=180] (b5);
\draw[e] (a4) to[out=0, in=180] (b2);

\draw[e] (b1) to[out=180, in=180] (b4);

\end{tikzpicture}
\quad
\cdot
\quad
\begin{tikzpicture}[x=1.5cm,y=-.5cm,baseline=-1.05cm]

\node[v] (b1) at (0,0) {};
\node[v] (b2) at (0,1) {};
\node[v] (b3) at (0,2) {};
\node[v] (b4) at (0,3) {};
\node[v] (b5) at (0,4) {};

\node[v] (c1) at (1,0) {};
\node[v] (c2) at (1,1) {};
\node[v] (c3) at (1,2) {};
\node[v] (c4) at (1,3) {};
\node[v] (c5) at (1,4) {};

\draw[e] (b5) to[out=0, in=180] (c4);

\draw[e] (b1) to[out=0, in=0] (b4);

\draw[e] (c1) to[out=180, in=180] (c3);

\end{tikzpicture}
\quad
=
\quad
\begin{tikzpicture}[x=1.5cm,y=-.5cm,baseline=-1.05cm]
\node[v] (a1) at (0,0) {};
\node[v] (a2) at (0,1) {};
\node[v] (a3) at (0,2) {};
\node[v] (a4) at (0,3) {};
\node[v] (a5) at (0,4) {};

\node[v] (b1) at (1,0) {};
\node[v] (b2) at (1,1) {};
\node[v] (b3) at (1,2) {};
\node[v] (b4) at (1,3) {};
\node[v] (b5) at (1,4) {};

\node[v] (c1) at (2,0) {};
\node[v] (c2) at (2,1) {};
\node[v] (c3) at (2,2) {};
\node[v] (c4) at (2,3) {};
\node[v] (c5) at (2,4) {};

\draw[e] (a2) to[out=0, in=0] (a3);

\draw[e] (a1) to[out=0, in=180] (b5);
\draw[e] (a4) to[out=0, in=180] (b2);

\draw[e] (b1) to[out=180, in=180] (b4);

\draw[e] (b5) to[out=0, in=180] (c4);

\draw[e] (b1) to[out=0, in=0] (b4);

\draw[e] (c1) to[out=180, in=180] (c3);

\end{tikzpicture}
\quad
= \delta \varepsilon \cdot
\quad
\begin{tikzpicture}[x=1.5cm,y=-.5cm,baseline=-1.05cm]
\node[v] (a1) at (0,0) {};
\node[v] (a2) at (0,1) {};
\node[v] (a3) at (0,2) {};
\node[v] (a4) at (0,3) {};
\node[v] (a5) at (0,4) {};

\node[v] (c1) at (1,0) {};
\node[v] (c2) at (1,1) {};
\node[v] (c3) at (1,2) {};
\node[v] (c4) at (1,3) {};
\node[v] (c5) at (1,4) {};

\draw[e] (a2) to[out=0, in=0] (a3);

\draw[e] (a1) to[out=0, in=180] (c4);

\draw[e] (c1) to[out=180, in=180] (c3);

\end{tikzpicture}
\quad
\]
\end{center}

We get a single factor of $\delta$ from the single loop, and a single factor of $\varepsilon$ because there is a single contractible component (in the form of an isolated node) in the middle.

Note also that there are two left-to-right connections in $x$: one meets a left-to-right connection of $y$, forming a left-to-right connection in the product. The other meets a `missing' edge of $y$, and disappears in the product. In the concatenated diagram, one may think of contracting this edge down to its initial node. \end{example}

\begin{definition} \label{defTriv} The \emph{trivial module} $\t$ for $\rook \Brauer_n(\delta,\varepsilon)$ is a single copy of $R$, where diagrams with no missing edges and no left-to-left or right-to-right connections act as multiplication by 1, and other diagrams act as multiplication by 0. \end{definition}

Any subalgebra of $\rook \Brauer_n(\delta,\varepsilon)$ acts on $\t$ by restriction. To state our main result (Theorem \ref{subalgebraOfBrauer}), we need some more terminology, which was introduced by Ridout and Saint-Aubin \cite{RidoutSaintAubin} in the context of Temperley-Lieb algebras.

\begin{definition} \label{defLS} By slicing vertically down the middle of a rook-Brauer diagram we obtain two `half-diagrams', which are called the \emph{left and right link states} of the diagram. Intrinsically, a link state consists of $n$ nodes, arranged vertically, where each node may be connected by an edge to at most one other, and a node not connected to any other may (or may not) have a `hanging' edge attached, whose other end is not connected to anything. A hanging edge (or the node to which it is attached) is called a \emph{defect}. \end{definition}

In particular, we do not think of link states as having an intrinsic handedness, even though they will always be drawn with one. This means for example that we can say that `a diagram that is symmetric about the vertical has identical left and right link states'.

\begin{remark} \label{rmk:cellVsLS} Our link states are related to the part of the data defining a cellular algebra that Graham and Lehrer \cite{GrahamLehrer} call $M(\lambda)$. For the Temperley-Lieb algebras (Definition \ref{defTL}) they coincide with the sets $M(\lambda)$ of Graham and Lehrer's cell structure \cite[Theorem 6.7]{GrahamLehrer}, but for the Brauer algebra (Definition \ref{defBrauer}) they do not \cite[Theorem 4.10]{GrahamLehrer}. In basic terms, the reason is that a Temperley-Lieb diagram may be recovered from its link states, but a Brauer diagram may not, since one has lost information, namely the permutation on the left-to-right connections. For cellular algebras, this cannot happen: in Graham and Lehrer's notation, a pair of elements $S,T \in M(\lambda)$ uniquely determines a basis element $C_{S,T}^{\lambda}$. As such, Graham and Lehrer's cellular basis for the Brauer algebra is \emph{not} the basis of diagrams.

\emph{As such, our link states should be thought of a naive diagrammatic generalisation of Ridout and Saint-Aubin's definition to rook-Brauer algebras.} We should note that our point of view on the Brauer algebras is implicit in \cite{GrahamLehrer}: on the way to their cell structure on the Brauer algebra, they study the basis of Brauer diagrams, and denote basis elements by $[S_1,S_2,w]$, where the $S_i$ are link states in our sense, and $w$ is the permutation on the through strands. In K\"onig and Xi's point of view on cellular algebras, there is also a sense in which `our point of view may be refined to a cellular structure', and this is made precise as \cite[Corollary 5.1]{KonigXi2}.
\end{remark}

Although the name `link state' was introduced by Ridout and Saint Aubin in \cite{RidoutSaintAubin} those authors point out that the concept predates that paper, being equivalent to the \emph{parenthesis structures} of \cite{Kauffman}, and the \emph{arch configurations} of \cite{FGG-TL}, as well as the cellular structure mentioned in Remark \ref{rmk:cellVsLS} above. Link states are such natural combinatorial objects that they surely occur in many other places under many other names.

\begin{example} Here is an illustration of the function that carries a rook-Brauer diagram to its right link state, as applied to a diagram in $\rook \Brauer_5$:

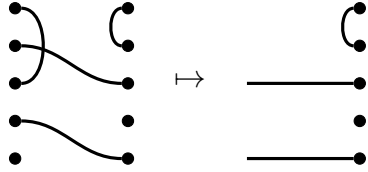
\begin{center}
\[
\begin{tikzpicture}[x=1.5cm,y=-.5cm,baseline=-1.05cm]
\node[v] (a1) at (0,0) {};
\node[v] (a2) at (0,1) {};
\node[v] (a3) at (0,2) {};
\node[v] (a4) at (0,3) {};
\node[v] (a5) at (0,4) {};

\node[v] (b1) at (1,0) {};
\node[v] (b2) at (1,1) {};
\node[v] (b3) at (1,2) {};
\node[v] (b4) at (1,3) {};
\node[v] (b5) at (1,4) {};

\draw[e] (a1) to[out=0, in=0] (a3);

\draw[e] (a2) to[out=0, in=180] (b3);
\draw[e] (a4) to[out=0, in=180] (b5);

\draw[e] (b1) to[out=180, in=180] (b2);

\end{tikzpicture}
\quad
\mapsto
\quad
\begin{tikzpicture}[x=1.5cm,y=-.5cm,baseline=-1.05cm]

\node[v] (b1) at (1,0) {};
\node[v] (b2) at (1,1) {};
\node[v] (b3) at (1,2) {};
\node[v] (b4) at (1,3) {};
\node[v] (b5) at (1,4) {};

\draw[e] (0,2) to[out=0, in=180] (b3);

\draw[e] (b1) to[out=180, in=180] (b2);
\draw[e] (b5) to[out=180, in=0] (0,4);

\end{tikzpicture}
\quad
\]
\end{center}

Since the original diagram had two left-to-right connections, the right link state has two defects.
\end{example}

For $i = 0, \dots , n$, the $R$-span of the rook-Brauer $n$-diagrams having at most $i$ left-to-right connections is a two-sided ideal of $\rook \Brauer_n(\delta, \varepsilon)$. We denote this ideal by $I_i$. For convenience, we set $I_{-1}=0$. Let $P_i$ be the set of link states with precisely $i$ defects.

Link states are important because they provide a natural family of one-sided ideals. First, given a link state, we may perform the following two operations:
\begin{itemize}
    \item Connect two defects to form an edge. For shorthand, we call this operation a \emph{\mOne}, and it defines a function $P_i \to P_{i-2}$.
    \item Delete a defect, leaving a missing edge. We call this operation a \emph{\mTwo}, and it defines a function $P_i \to P_{i-1}$.
\end{itemize}

For $p \in P_i$, let $J_p$ be the $R$-submodule of $\rook \Brauer_n(\delta, \varepsilon)$ with basis the diagrams having right link state obtained from $p$ by a (perhaps empty) sequence of {\mOne}s and {\mTwo}s. The following observation justifies the introduction of this terminology. Given diagrams $x$ and $y$, the right link state of $xy$ (ignoring factors of $\delta$ and $\varepsilon$) must be obtained from the right link state of $y$ by precisely such a sequence. This implies that $J_p$ is actually a left ideal of $\rook \Brauer_n(\delta, \varepsilon)$.

We now state the main theorem.

\begin{theorem} \label{subalgebraOfBrauer} Let $0 \leq \ell \leq m \leq n-1$. Let $A$ be a subalgebra of $\rook \Brauer_n(\delta, \varepsilon)$, such that \begin{itemize}
    \item as an $R$-module, $A$ is free on a subset of the rook-Brauer $n$-diagrams, and
    \item for $i$ in the range $\ell \leq i \leq m$, for each link state $p \in P_i$, if $A$ contains at least one diagram with right link state $p$, then $A$ contains an idempotent $e_p$ such that in $A$ we have equality of left ideals $A \cdot e_p = A \cap J_p$.
\end{itemize}
Then we have a chain of isomorphisms $$\Tor^{\faktor{A}{A \cap I_{\ell-1}}}_*(\t,\t) \cong \Tor^{\faktor{A}{A \cap I_{\ell}}}_*(\t,\t) \cong \dots \cong \Tor^{\faktor{A}{A \cap I_{m}}}_*(\t,\t).$$ \end{theorem}

Another way of stating the second hypothesis is to say that for each $p \in P_i$, the left ideal $A \cap J_p < A$ is either empty or is principal and generated by an idempotent.

We will only use this theorem when $\ell=0$ or $1$ and $m = n-1$, and in fact almost always with $\ell = 0$. The case $\ell = 1$ will be useful basically only for explaining what even-parity Sroka-type results should mean (Theorems \ref{generalisedSroka} and \ref{generalisedBrauerSroka}).

In the most common case, $\ell=0$, since $I_{-1}=0$ by definition, the first group in this chain of isomorphisms is actually the homology of $A$ itself. The theorem then says that the homology of $A$ is isomorphic to the homology of the subalgebra of $A$ generated by diagrams with the maximum number $n$ of left-to-right connections, since (Proposition \ref{retract}) this is isomorphic to $\faktor{A}{A \cap I_{n-1}}$. This subalgebra is typically well understood: for our examples it is usually the symmetric group algebra $R \Sigma_n$, with the sole exception of the Temperley-Lieb algebras, when it is just a copy of the ground ring $R$. The homology of $R \Sigma_n$ is the same as the group homology of $\Sigma_n$, which was computed by Nakaoka \cite{Nakaoka}. The homology of $R$ is trivial - for example by thinking of $R$ as the group ring of the trivial group. In particular, the right hand end of the above chain of isomorphisms is supposed to be useful for understanding the left hand end, rather than vice versa.

\subsection{Sketch of methods}

In this subsection, fix a commutative ring $R$ and an $R$-algebra $A$ with a choice of trivial module $\t$. We will sketch the proof of Theorem \ref{quotienting}, which is to say, the proof of Theorem \ref{quotientingPlus} with $k=1$. This is the simplest version of our main result.

Begin with the following observation. For $I$ a left ideal of $A$, if $I$ is generated by commuting idempotents, then $I$ is projective as a left $A$-module, and if $\t \cdot I = 0$ then $\t \otimes_A I = 0$. Since $I$ is projective as a left $A$-module, $$\dots \to 0 \to 0 \to I \to A \to \faktor{A}{I}$$ is a `very short' (augmented) projective resolution of $\faktor{A}{I}$. These facts suffice to show that if $\t \cdot I = 0$ then $\Tor_q^A(\t, \faktor{A}{I})=0$ for $q>0$.

If $I$ is a two-sided ideal of $A$, then $\faktor{A}{I}$ is again an algebra. If $I$ acts trivially on $\t$ in the sense that $\t \cdot I = I \cdot \t = 0$, then $\t$ becomes an $\faktor{A}{I}$-bimodule. Under these circumstances, the homology of $\faktor{A}{I}$ is well-defined, and it makes sense to ask when we have a homology isomorphism $$\Tor_*^A(\t,\t) \cong \Tor_*^{\faktor{A}{I}}(\t,\t).$$ The above $\Tor$-vanishing then implies that $\otimes_A \faktor{A}{I}$ carries free $A$-resolutions of $\t$ to free $\faktor{A}{I}$-resolutions of $\t$. This is the key step involved in showing that $\Tor_*^{\faktor{A}{I}}(\t,\t)$ and $\Tor_*^A(\t,\t)$ are the homology of the same chain complex, hence are isomorphic. This concludes the sketch of the proof of Theorem \ref{quotienting}.

Taking $I=A \cap I_{n-1}$, the ideal of diagrams having fewer than the maximum number of connections, this suffices for our two simplest applications, Theorems \ref{rookInvertible} and Theorem \ref{rookBrauerInvertible}. As has been alluded to, however, these are the only two theorems for which this simplest result will suffice.

The sense in which reality is not usually so simple is that the generators we identify for $A \cap I_{n-1}$ do not usually commute. To deal with this, we work inductively, in a sense which is precisely parallel to that introduced by Boyd and Hepworth \cite{BH}: here we will outline the sense in which each $I_i$ is still (inductively) projective.

From our point of view, working inductively is made possible by the observation (which occurs in the proof of Theorem \ref{subalgebraOfBrauer}) that for distinct link states $p$ and $q$ having $i$ defects, we have the inclusion $J_p \cap J_q \subset I_{i-1}$. The fact that every diagram has a link state amounts to a sum $$\faktor{I_i}{I_{i-1}} = \sum_{p \in P_i} \faktor{A \cap J_p}{I_{i-1}},$$ regarded as internal to $\faktor{A}{I_{i-1}}$, and the intersection condition says that actually this sum is direct. Since a direct sum of projectives is projective, we may focus our attention on establishing (the inductive fact) that each $\faktor{A \cap J_p}{I_{i-1}}$ is either zero or principal idempotent. This often \emph{can} be shown in practice - and it is this which makes Theorem \ref{subalgebraOfBrauer} work.

Once we know that $\faktor{I_i}{I_{i-1}}$ is projective over $\faktor{A}{I_{i-1}}$, we may proceed as before, and the resulting resolutions of $\faktor{A}{I_{i}}$ over $\faktor{A}{I_{i-1}}$ play precisely the same role as those constructed in \cite{BH} and \cite{BHP}. In those papers, a spectral sequence argument is required in the induction that then shows the $\Tor$-vanishing, but the relative simplicity of our resolutions eliminates the need for this.

\subsection{Discussion and philosophy}

This paper began life as an attempt to establish homological stability for rook algebras, and became something quite different. As such, we make no attempt at a `true' homological stability result for not-necessarily invertible parameter rook algebras - our methods do not apply.

It should not go unremarked that all of the algebras discussed here are isomorphic to their opposite. Geometrically, this corresponds to reflecting diagrams in the vertical. As such, all results of this paper still hold after interchanging `right' and `left' everywhere.

The verification of Theorem \ref{subalgebraOfBrauer} in this paper always goes via Lemma \ref{mirrorDiagram} (for invertible-parameter applications) or Lemma \ref{LSControl} (for Sroka-type applications). These lemmas verify a more eigenspace-flavoured statement, which echoes the cellular point of view, and which some readers may find helpful.

We hope that this paper succeeds in making precise the following morals:
\begin{itemize}
    \item Diagram algebras come with a `group in the middle' (Proposition \ref{retract}) which is independent of the parameters, always spans a subalgebra which is a retract, and often turns out to know all of the homology.
    \item Sroka's Theorem \cite{Sroka} is just one instance of a class of results (Theorems \ref{generalisedBrauerSroka} and \ref{generalisedSroka}) which might reasonably be expected to hold for other families of algebras.
\end{itemize}

The following ideas are borrowed from the study of the (real or complex) Temperley-Lieb algebra in other areas, but all of them are applicable to all of the diagram algebras we consider.
\begin{itemize}
    \item Link states (Definition \ref{defLS}): These appear for Temperley-Lieb algebras in \cite{RidoutSaintAubin}, and under other names in \cite{Kauffman}, \cite{FGG-TL}, and, in the context of cellular algebras, \cite{GrahamLehrer}. We repeat the warning (Remark \ref{rmk:cellVsLS}) that for Brauer algebras our link states do not coincide with the cellular structure. Link states define the family of ideals $J_p$ whose structural properties make possible Theorem \ref{subalgebraOfBrauer}. There will be echoes of various related ideas from \cite{GrahamLehrer} and \cite{RidoutSaintAubin}, notably the cell modules/link modules, and their inner products. We note when these parallels appear.
    \item Meanders \cite{FGG}: We will not use this language explicitly, but the basic idea (Lemma \ref{hardSingleTrundle}) needed to apply Theorem \ref{subalgebraOfBrauer} in the proof of our generalisation of Sroka's result (Theorem \ref{generalisedSroka}) is a procedure due to Francesco, Golinelli, and Guitter in \cite{FGG}. For us, this amounts to the possibility of making a `planar Eulerian cycle' through the edges of a given planar link state. In their language, one says that `every upper arch configuration extends to a one-component meander'. The analogous result for Brauer algebras is much simpler, and this is ultimately the reason why the proof of Theorem \ref{generalisedBrauerSroka} is easier.
\end{itemize}

Lastly, it seems appropriate to comment on the applicability of the cellular point of view.

\begin{remark} \label{rmk:cellularApplicability} The construction of suitable idempotents is a familiar thing in the context of cellular algebras; cellular algebras typically have many idempotents. For example, over a field, cell ideals in K\"onig and Xi's sense are either square zero or are generated by an idempotent \cite[Proposition 4.1]{KonigXi1}.

The main difficulty with applying the cellular point of view in our context is that (assuming the language of \cite{GrahamLehrer}) the nondegeneracy-type conditions on the bilinear forms $\phi_{\lambda}$ that are necessary to guarantee the existence of suitable idempotents often fail. For example, Graham and Lehrer show that even for their cell structure on the Brauer algebra, over a field of positive characteristic, $\phi_\lambda$ typically vanishes for at least some $\lambda$ \cite[Theorem 4.17]{GrahamLehrer}. This suggests that for example even Theorem \ref{BrauerRecovery} (the invertible-parameter Brauer result of Boyd-Hepworth-Patzt) cannot be recovered using this cellular structure.

Intuitively, then, the sensitivity of the cellular structure to the ground ring, which was an advantage in classifying irreducibles (where one wishes to determine for which $\lambda$ the bilinear form $\phi_{\lambda}$ vanishes \cite[Theorem 3.4]{GrahamLehrer}) becomes a disadvantage in attempting to compute homology (where one would wish to show that each $\phi_{\lambda}$ satisfied some nondegeneracy-type condition, and in particular did not vanish). Especially, results on homology of diagram algebras have tended to be very robust to the ground ring, depending at most on whether the parameter $\delta$ is invertible, and not, for example, on whether $R$ is a field, so it seems likely that technology which is less sensitive to the ground ring is more appropriate.
\end{remark}

\subsection{Organisation} This paper is organised as follows. In Section \ref{sectionAlgs} we define the families of algebras in which we are interested. In Section \ref{sectionRetract} we discuss the {\maxTpt} subalgebra, recording its main properties as Proposition \ref{retract}. This section is not much used in what follows, but provides context for the results. In Section \ref{SectionIdempotents} we will prove Theorem \ref{quotientingPlus}, and deduce Corollary \ref{quotientingInductively}. These belong to pure algebra and are the first steps towards Theorem \ref{subalgebraOfBrauer}. Theorem \ref{quotientingPlus} is already good enough for our applications to rook algebras, which we then give in Section \ref{sectionRook}. In Section \ref{SectionDiagram} we prove Theorem \ref{subalgebraOfBrauer}, our main result.

The remainder of the paper then deals with applications. Invertible-parameter applications are given in Section \ref{applicationsInvertible}: to the (rook-)Brauer algebras in Subsection \ref{subsectionInvBrauer}, and to the Temperley-Lieb algebras in Subsection \ref{subsectionInvTL}. Sroka-type applications are then given in Section \ref{applicationsSroka}, with applications to the (rook-)Brauer algebras in Subsection \ref{subsectionSrokaBrauer}, and applications to the Temperley-Lieb algebras in Subsection \ref{subsectionSrokaTL}. The applications are all independent of one another, though the invertible-parameter applications all use Lemma \ref{mirrorDiagram}, and the Sroka-type applications all use Lemma \ref{LSControl}. They are arranged roughly in order of increasing complexity. Many of the ideas introduced in later applications could be used in earlier ones, but we prefer to give simple proofs where possible, invoking the more complicated definitions only when necessary.

\subsection{Acknowledgements} I would like to thank Rachael Boyd and Richard Hepworth, for their encouragement, and for their suggestions and comments on previous drafts. I would also like to thank Sam Hughes for early encouragement and feedback, and Lawk Mineh, Niall Taggart, and an anonymous reviewer for useful comments on organisation and presentation. I would also like to thank Rachael for suggesting rook algebras, a class which leads quite naturally to thinking in terms of idempotents, and to thank Richard for giving an excellent and accessible talk on homological stability of algebras at Ran Levi's $60^{\textrm{th}}$ birthday conference. Thanks are therefore also due to certain bureaucratic forces which made it possible for me to attend that conference. This work was supported by the European Research council (ERC) through Gijs Heuts' grant “Chromatic homotopy theory of spaces”, grant no. 950048.

\section{The algebras} \label{sectionAlgs}

Here, we collect the definitions of the subalgebras of $\rook \Brauer_n(\delta, \varepsilon)$ that will appear in the applications of Theorem \ref{subalgebraOfBrauer}. Each of them will arise by restricting the set of permitted diagrams. As noted earlier, each of these algebras has a trivial module, obtained by restriction of scalars on the trivial module for $\rook \Brauer_n(\delta, \varepsilon)$.

For each subalgebra $A$, the subalgebra of $A$ on diagrams having the maximum number $n$ of left-to-right connections is a retract of $A$ (Proposition \ref{retract}). In each case, we will identify this retract - for want of a better name, we will call it the \emph{{\maxTpt} subalgebra} of $A$. It is always the group algebra on a subgroup $G$ of $\Sigma_n$, and is always independent of $\delta$ and $\varepsilon$. This tells us what answer to expect from an application of Theorem \ref{subalgebraOfBrauer} to the algebra in question. More immediately, since all functors preserve retracts, this identifies a summand of the homology of $A$ very cheaply.

First, for the rook-Brauer algebra itself, the {\maxTpt} subalgebra is isomorphic to the symmetric group algebra $R \Sigma_n$ (Proposition \ref{symmetricGroup}). Briefly, a diagram with the maximum number of left-to-right connections gives a bijection between the nodes on the left and those on the right, hence a permutation, and each permutation can be represented by such a diagram. This identification between permutations and diagrams respects the multiplication, so we see directly that this subalgebra is independent of $\delta$ and $\varepsilon$.

\subsection{The rook algebras}

The origin of the name `rook' appears to be with Solomon \cite{Solomon}, who defines the \emph{rook monoid} on $n$ strands to be the set of partially defined permutations of $n$ letters, together with composition. A diagram with no left-to-left or right-to-right connections may be interpreted as a partial permutation: the permutation is defined on those nodes on the left that have a left-to-right connection, and these are mapped to the node at the other end of the connection. The following definition therefore includes the algebra on Solomon's rook monoid as the case $\varepsilon=1$.

\begin{definition} \label{defRook} A \emph{rook} $n$-diagram is a rook-Brauer $n$-diagram having no left-to-left or right-to-right connections. The \emph{rook algebra} $\rook_n = \rook_n(\varepsilon)$ (determined by a ring $R$, a natural number $n$, and a parameter $\varepsilon \in R$) is the subalgebra of $\rook \Brauer_n(\delta, \varepsilon)$ on the rook diagrams. \end{definition}

The rook algebra is independent of $\delta$ (which is therefore omitted from the notation) since the absence of left-to-left and right-to-right connections means that multiplication can yield no loops.

Diagrammatically, this means we have the diagrams obtained from those in the group algebra $R \Sigma_n$ of the symmetric group by deleting strands, and a sample multiplication in $\rook_5$ is:
\begin{center}
\[
\begin{tikzpicture}[x=1.5cm,y=-.5cm,baseline=-1.05cm]
\node[v] (a1) at (0,0) {};
\node[v] (a2) at (0,1) {};
\node[v] (a3) at (0,2) {};
\node[v] (a4) at (0,3) {};
\node[v] (a5) at (0,4) {};

\node[v] (b1) at (1,0) {};
\node[v] (b2) at (1,1) {};
\node[v] (b3) at (1,2) {};
\node[v] (b4) at (1,3) {};
\node[v] (b5) at (1,4) {};

\draw[e] (a1) to[out=0, in=180] (b3);
\draw[e] (a2) to[out=0, in=180] (b2);
\draw[e] (a5) to[out=0, in=180] (b4);

\end{tikzpicture}
\quad
\cdot
\quad
\begin{tikzpicture}[x=1.5cm,y=-.5cm,baseline=-1.05cm]

\node[v] (b1) at (0,0) {};
\node[v] (b2) at (0,1) {};
\node[v] (b3) at (0,2) {};
\node[v] (b4) at (0,3) {};
\node[v] (b5) at (0,4) {};

\node[v] (c1) at (1,0) {};
\node[v] (c2) at (1,1) {};
\node[v] (c3) at (1,2) {};
\node[v] (c4) at (1,3) {};
\node[v] (c5) at (1,4) {};

\draw[e] (b1) to[out=0, in=180] (c1);
\draw[e] (b2) to[out=0, in=180] (c4);
\draw[e] (b4) to[out=0, in=180] (c2);

\end{tikzpicture}
\quad
=
\quad
\begin{tikzpicture}[x=1.5cm,y=-.5cm,baseline=-1.05cm]

\node[v] (a1) at (0,0) {};
\node[v] (a2) at (0,1) {};
\node[v] (a3) at (0,2) {};
\node[v] (a4) at (0,3) {};
\node[v] (a5) at (0,4) {};

\node[v] (b1) at (1,0) {};
\node[v] (b2) at (1,1) {};
\node[v] (b3) at (1,2) {};
\node[v] (b4) at (1,3) {};
\node[v] (b5) at (1,4) {};

\node[v] (c1) at (2,0) {};
\node[v] (c2) at (2,1) {};
\node[v] (c3) at (2,2) {};
\node[v] (c4) at (2,3) {};
\node[v] (c5) at (2,4) {};

\draw[e] (a1) to[out=0, in=180] (b3);
\draw[e] (a2) to[out=0, in=180] (b2);
\draw[e] (a5) to[out=0, in=180] (b4);

\draw[e] (b1) to[out=0, in=180] (c1);
\draw[e] (b2) to[out=0, in=180] (c4);
\draw[e] (b4) to[out=0, in=180] (c2);

\end{tikzpicture}
\quad
= \varepsilon \cdot
\quad
\begin{tikzpicture}[x=1.5cm,y=-.5cm,baseline=-1.05cm]

\node[v] (a1) at (0,0) {};
\node[v] (a2) at (0,1) {};
\node[v] (a3) at (0,2) {};
\node[v] (a4) at (0,3) {};
\node[v] (a5) at (0,4) {};

\node[v] (c1) at (1,0) {};
\node[v] (c2) at (1,1) {};
\node[v] (c3) at (1,2) {};
\node[v] (c4) at (1,3) {};
\node[v] (c5) at (1,4) {};

\draw[e] (a2) to[out=0, in=180] (c4);
\draw[e] (a5) to[out=0, in=180]   (c2);

\end{tikzpicture}
\]
\end{center}
where we have only a single factor of $\varepsilon$ because there was only a single isolated point in the middle.

The {\maxTpt} subalgebra of the rook algebra, as for the rook-Brauer algebra, is $R \Sigma_n$ itself.

\subsection{The Brauer algebras}

The Brauer algebras are distinguished by having no missing edges. They substantially predate the rook-Brauer algebra, originating in a 1937 paper of Brauer \cite{Brauer}, and their homology was studied by Boyd, Hepworth, and Patzt in \cite{BHP}.

\begin{definition} \label{defBrauer} A \emph{Brauer} $n$-diagram is a rook-Brauer $n$-diagram having no missing edges. The \emph{Brauer algebra} $\Brauer_n = \Brauer_n(\delta)$ (determined by a ring $R$, a natural number $n$, and a parameter $\delta \in R$) is the subalgebra of $\rook \Brauer_n(\delta, \varepsilon)$ on the Brauer diagrams. \end{definition}

Note that the absence of missing edges means that after concatenation of diagrams, each node in the middle has valence 2, so all connected components in the middle must be loops, and multiplication can yield no factors of $\varepsilon$. The Brauer algebra is therefore independent of $\varepsilon$, justifying the notation. A sample multiplication in $\Brauer_5$ is:

\begin{center}
\[
\begin{tikzpicture}[x=1.5cm,y=-.5cm,baseline=-1.05cm]
\node[v] (a1) at (0,0) {};
\node[v] (a2) at (0,1) {};
\node[v] (a3) at (0,2) {};
\node[v] (a4) at (0,3) {};
\node[v] (a5) at (0,4) {};

\node[v] (b1) at (1,0) {};
\node[v] (b2) at (1,1) {};
\node[v] (b3) at (1,2) {};
\node[v] (b4) at (1,3) {};
\node[v] (b5) at (1,4) {};

\draw[e] (a1) to[out=0, in=0] (a4);

\draw[e] (a5) to[out=0, in=180] (b1);
\draw[e] (a2) to[out=0, in=180] (b2);
\draw[e] (a3) to[out=0, in=180] (b3);

\draw[e] (b4) to[out=180, in=180] (b5);

\end{tikzpicture}
\quad
\cdot
\quad
\begin{tikzpicture}[x=1.5cm,y=-.5cm,baseline=-1.05cm]

\node[v] (b1) at (0,0) {};
\node[v] (b2) at (0,1) {};
\node[v] (b3) at (0,2) {};
\node[v] (b4) at (0,3) {};
\node[v] (b5) at (0,4) {};

\node[v] (c1) at (1,0) {};
\node[v] (c2) at (1,1) {};
\node[v] (c3) at (1,2) {};
\node[v] (c4) at (1,3) {};
\node[v] (c5) at (1,4) {};

\draw[e] (b5) to[out=0, in=180] (c5);

\draw[e] (b1) to[out=0, in=0] (b3);
\draw[e] (b2) to[out=0, in=0] (b4);

\draw[e] (c1) to[out=180, in=180] (c3);
\draw[e] (c2) to[out=180, in=180] (c4);

\end{tikzpicture}
\quad
=
\quad
\begin{tikzpicture}[x=1.5cm,y=-.5cm,baseline=-1.05cm]

\node[v] (a1) at (0,0) {};
\node[v] (a2) at (0,1) {};
\node[v] (a3) at (0,2) {};
\node[v] (a4) at (0,3) {};
\node[v] (a5) at (0,4) {};

\node[v] (b1) at (1,0) {};
\node[v] (b2) at (1,1) {};
\node[v] (b3) at (1,2) {};
\node[v] (b4) at (1,3) {};
\node[v] (b5) at (1,4) {};

\node[v] (c1) at (2,0) {};
\node[v] (c2) at (2,1) {};
\node[v] (c3) at (2,2) {};
\node[v] (c4) at (2,3) {};
\node[v] (c5) at (2,4) {};

\draw[e] (a1) to[out=0, in=0] (a4);

\draw[e] (a5) to[out=0, in=180] (b1);
\draw[e] (a2) to[out=0, in=180] (b2);
\draw[e] (a3) to[out=0, in=180] (b3);

\draw[e] (b4) to[out=180, in=180] (b5);

\draw[e] (b5) to[out=0, in=180] (c5);

\draw[e] (b1) to[out=0, in=0] (b3);
\draw[e] (b2) to[out=0, in=0] (b4);

\draw[e] (c1) to[out=180, in=180] (c3);
\draw[e] (c2) to[out=180, in=180] (c4);

\end{tikzpicture}
\quad
=
\quad
\begin{tikzpicture}[x=1.5cm,y=-.5cm,baseline=-1.05cm]

\node[v] (a1) at (0,0) {};
\node[v] (a2) at (0,1) {};
\node[v] (a3) at (0,2) {};
\node[v] (a4) at (0,3) {};
\node[v] (a5) at (0,4) {};

\node[v] (c1) at (1,0) {};
\node[v] (c2) at (1,1) {};
\node[v] (c3) at (1,2) {};
\node[v] (c4) at (1,3) {};
\node[v] (c5) at (1,4) {};

\draw[e] (a1) to[out=0, in=0] (a4);
\draw[e] (a3) to[out=0, in=0] (a5);

\draw[e] (a2) to[out=0, in=180] (b5);

\draw[e] (c1) to[out=180, in=180] (c3);
\draw[e] (c2) to[out=180, in=180] (c4);
\end{tikzpicture}
\]
\end{center}

As with the rook-Brauer and rook algebras, the {\maxTpt} subalgebra of the Brauer algebra is the symmetric group algebra $R \Sigma_n$.

\subsection{The Temperley-Lieb algebras}

The homology of the Temperley-Lieb algebras was studied by Boyd and Hepworth in \cite{BH} and \cite{BHComb}. This family is distinguished by being planar and having no missing edges, and originates in theoretical physics in the 1970s \cite{TemperleyLieb}.

\begin{definition} \label{defTL} A \emph{Temperley-Lieb} $n$-diagram is a rook-Brauer $n$-diagram that is planar and has no missing edges. The \emph{Temperley-Lieb algebra} $\TL_n = \TL_n(\delta)$ (determined by a ring $R$, a natural number $n$, and a parameter $\delta \in R$) is the subalgebra of $\rook \Brauer_n(\delta, \varepsilon)$ on the Temperley-Lieb diagrams. \end{definition}

In particular, this makes the Temperley-Lieb algebra a subalgebra of the Brauer algebra, and justifies the absence of $\varepsilon$ from the notation. A sample multiplication in $\TL_5$ is:
\begin{center}
\[
\begin{tikzpicture}[x=1.5cm,y=-.5cm,baseline=-1.05cm]
\node[v] (a1) at (0,0) {};
\node[v] (a2) at (0,1) {};
\node[v] (a3) at (0,2) {};
\node[v] (a4) at (0,3) {};
\node[v] (a5) at (0,4) {};

\node[v] (b1) at (1,0) {};
\node[v] (b2) at (1,1) {};
\node[v] (b3) at (1,2) {};
\node[v] (b4) at (1,3) {};
\node[v] (b5) at (1,4) {};

\draw[e] (a1) to[out=0, in=0] (a4);
\draw[e] (a2) to[out=0, in=0] (a3);

\draw[e] (a5) to[out=0, in=180] (b1);

\draw[e] (b2) to[out=180, in=180] (b3);
\draw[e] (b4) to[out=180, in=180] (b5);

\end{tikzpicture}
\quad
\cdot
\quad
\begin{tikzpicture}[x=1.5cm,y=-.5cm,baseline=-1.05cm]

\node[v] (b1) at (0,0) {};
\node[v] (b2) at (0,1) {};
\node[v] (b3) at (0,2) {};
\node[v] (b4) at (0,3) {};
\node[v] (b5) at (0,4) {};

\node[v] (c1) at (1,0) {};
\node[v] (c2) at (1,1) {};
\node[v] (c3) at (1,2) {};
\node[v] (c4) at (1,3) {};
\node[v] (c5) at (1,4) {};

\draw[e] (b1) to[out=0, in=180] (c3);
\draw[e] (b2) to[out=0, in=180] (c4);
\draw[e] (b5) to[out=0, in=180] (c5);

\draw[e] (b3) to[out=0, in=0] (b4);

\draw[e] (c1) to[out=180, in=180] (c2);
\end{tikzpicture}
\quad
=
\quad
\begin{tikzpicture}[x=1.5cm,y=-.5cm,baseline=-1.05cm]

\node[v] (a1) at (0,0) {};
\node[v] (a2) at (0,1) {};
\node[v] (a3) at (0,2) {};
\node[v] (a4) at (0,3) {};
\node[v] (a5) at (0,4) {};

\node[v] (b1) at (1,0) {};
\node[v] (b2) at (1,1) {};
\node[v] (b3) at (1,2) {};
\node[v] (b4) at (1,3) {};
\node[v] (b5) at (1,4) {};

\node[v] (c1) at (2,0) {};
\node[v] (c2) at (2,1) {};
\node[v] (c3) at (2,2) {};
\node[v] (c4) at (2,3) {};
\node[v] (c5) at (2,4) {};

\draw[e] (a1) to[out=0, in=0] (a4);
\draw[e] (a2) to[out=0, in=0] (a3);

\draw[e] (a5) to[out=0, in=180] (b1);

\draw[e] (b2) to[out=180, in=180] (b3);
\draw[e] (b4) to[out=180, in=180] (b5);

\draw[e] (b1) to[out=0, in=180] (c3);
\draw[e] (b2) to[out=0, in=180] (c4);
\draw[e] (b5) to[out=0, in=180] (c5);

\draw[e] (b3) to[out=0, in=0] (b4);

\draw[e] (c1) to[out=180, in=180] (c2);

\end{tikzpicture}
\quad
=
\quad
\begin{tikzpicture}[x=1.5cm,y=-.5cm,baseline=-1.05cm]

\node[v] (a1) at (0,0) {};
\node[v] (a2) at (0,1) {};
\node[v] (a3) at (0,2) {};
\node[v] (a4) at (0,3) {};
\node[v] (a5) at (0,4) {};

\node[v] (c1) at (1,0) {};
\node[v] (c2) at (1,1) {};
\node[v] (c3) at (1,2) {};
\node[v] (c4) at (1,3) {};
\node[v] (c5) at (1,4) {};

\draw[e] (a1) to[out=0, in=0] (a4);
\draw[e] (a2) to[out=0, in=0] (a3);

\draw[e] (a5) to[out=0, in=180] (c3);

\draw[e] (c1) to[out=180, in=180] (c2);
\draw[e] (c4) to[out=180, in=180] (c5);
\end{tikzpicture}
\]
\end{center}

This particular multiplication produces no factors of $\delta$. The possibility of constructing diagrams that multiply to give no factors of $\delta$ will be key to our Sroka-type applications.

Planarity makes the {\maxTpt} subalgebra much smaller: the only planar diagram having $n$ left-to-right connections is the identity diagram (where all edges are horizontal). The {\maxTpt} subalgebra of $\TL_n$ is therefore just a copy of $R$.

\section{The canonical group algebra retract} \label{sectionRetract}

In this section we record Proposition \ref{retract}, an observation that we used frequently in the introduction. Recall that $R$ is always assumed to be a commutative ring with unit, and $I_{n-1} \leq \rook \Brauer_n(\delta, \varepsilon)$ is the ideal with $R$-basis consisting of diagrams having fewer than $n$ left-to-right connections.

We begin with the observation that $\rook \Brauer_n(\delta, \varepsilon)$ contains a copy of the group algebra of the symmetric group. The proof is standard, hence omitted, but is a good exercise.

\begin{proposition} \label{symmetricGroup} For any $\delta, \varepsilon \in R$, and any $n$, the set $S_{\max}$ of rook-Brauer $n$-diagrams with the maximum number $n$ of left-to-right connections forms a multiplicatively closed subset of $\rook \Brauer_n(\delta, \varepsilon)$. Under the multiplication, this subset is canonically isomorphic to the symmetric group $\Sigma_n$. \end{proposition}

The key point is that if $x$ and $y$ are diagrams having the maximum number $n$ of left-to right connections then the product $xy$ again has the maximum number $n$ of left-to-right connections. This closure property does not hold in general for the set of diagrams having at least $j$ left-to-right connections - it is special to $j=n$.

The intersection of two multiplicatively closed subsets of $\rook \Brauer_n(\delta, \varepsilon)$ is again a multiplicatively closed subset. Thus, if some subalgebra $A$ of $\rook \Brauer_n(\delta, \varepsilon)$ has non-empty intersection with $S_{\max}$, then in fact $A \cap S_{\max}$ is canonically identified with a subgroup $G \leq \Sigma_n$ under the isomorphism of Proposition \ref{symmetricGroup}. In other words, we have the following corollary.

\begin{corollary} \label{subgroup} Let $A$ be a subalgebra of $\rook \Brauer_n(\delta, \varepsilon)$, which is free on a subset of the rook-Brauer diagrams, and contains at least one diagram having the maximum number $n$ of left-to-right connections. The set of diagrams in $A$ having the maximum number $n$ of left-to-right connections is multiplicatively closed, and canonically isomorphic to some subgroup $G \leq \Sigma_n$. \qed \end{corollary}

Taking the $R$-span of $A \cap S_{\max}$ therefore identifies a canonical subalgebra of $A$ as a group algebra. This canonical group-subalgebra of $A$ is actually a retract, as follows.

\begin{proposition} \label{retract} Let $A$ be a subalgebra of $\rook \Brauer_n(\delta, \varepsilon)$, which is free on a subset of the rook-Brauer diagrams, and contains at least one diagram having the maximum number $n$ of left-to-right connections. Let $A_{\max}$ denote the $R$-span of the diagrams in $A$ with the maximum number $n$ of left-to-right connections. The following properties are satisfied:
\begin{itemize}
    \item $A_{\max}$ is a subalgebra of $A$, and is canonically isomorphic to the group ring $RG$ of some $G \leq \Sigma_n$.
    \item $A_{\max}$ is isomorphic to the quotient $\faktor{A}{A \cap I_{n-1}}$, hence is a retract of $A$.
\end{itemize} \end{proposition}

\begin{proof} To see the first claim, note that $A_{\max}$ is free on the set of diagrams in $A$ having the maximum number of left-to-right connections, hence is isomorphic to the required $RG$ by Corollary \ref{subgroup}.

To see the second claim, note that since $A$ is assumed to be free on a subset of the diagrams, $A \cap I_{n-1}$ has $R$-basis those diagrams in $A$ with fewer than $n$ left-to-right connections. The quotient $\faktor{A}{A \cap I_{n-1}}$ therefore has a basis given by the diagrams with precisely $n$ left-to-right connections, so the composite $A_{\max} \to A \to \faktor{A}{A \cap I_{n-1}}$ is an isomorphism, so $A_{\max} \cong \faktor{A}{A \cap I_{n-1}}$ is a retract of $A$, as required.
\end{proof}

\section{Idempotents and Ideals} \label{SectionIdempotents}

In this section, we will prove Theorem \ref{quotientingPlus}, and deduce Corollary \ref{quotientingInductively}, which is essentially an inductive version of the same thing. The word `inductive' here refers to the same basic phenomenon as in the Boyd-Hepworth \emph{inductive resolutions} of \cite{BH} and \cite{BHP}. These results belongs to pure algebra, and our main result (Theorem \ref{subalgebraOfBrauer}) is essentially the specialisation of Corollary \ref{quotientingInductively} to diagram algebras.

In this section, $R$ is a commutative unital ring, and all algebras are associative and unital, but need not be augmented. The next two lemmas are well-known (see for example \cite[Chapter 1, Exercise 2.4]{Liu}) but we include proofs because they are quite simple.

\begin{lemma} Let $A$ be an $R$-algebra, and let $e \in A$ be idempotent ($e^2=e$). Then the left ideal $A e$ generated by $e$ is a projective $A$-module. \label{singleIdempotent} \end{lemma} 

\begin{proof} The point is that the sequence of left $A$-modules $$0 \to Ae \to A \xrightarrow{\cdot(1-e)} A(1-e) \to 0$$ is split short exact. Injectivity and surjectivity are clear. To see that the composite is zero, note that an element of $Ae$ is of the form $\alpha e$ for $\alpha \in A$, and idempotency of $e$ then gives $$\alpha e \cdot (1-e) = \alpha (e - e^2) = \alpha (e-e)=0.$$ For exactness, if $\alpha(1-e)=0$, then $\alpha = \alpha e$ lies in $Ae$. For splitness, the first map is split by $\cdot e$, or alternatively the second map is split by $\cdot (1+e)$.

Splitness gives that $Ae$ is a direct summand of (the free left $A$-module) $A$, hence is projective, as required. 
\end{proof}

This lemma can be strengthened to cover ideals generated by commuting families of idempotents, as follows.

\begin{corollary} Let $A$ be an $R$-algebra, and let $J$ be a left ideal of $A$ generated by finitely many commuting idempotents $e_1, \dots, e_n$. Then $J$ is a projective $A$-module. \label{idempotency} \end{corollary}

\begin{proof}  If $e_1$ and $e_2$ are commuting idempotents, then $e_1 + e_2 - e_1e_2$ is idempotent, and we have $$e_1 \cdot(e_1 + e_2 - e_1e_2) = e_1 + e_1 e_2 - e_1 e_2 = e_1,$$ and $$e_2 \cdot(e_1 + e_2 - e_1e_2) = e_2 e_1 + e_2 - e_2 e_1 e_2 = e_2 e_1 + e_2 - e_2 e_1 = e_2.$$

It follows inductively that a left ideal generated by commuting idempotents is actually generated by a single idempotent.

The result then follows from Lemma \ref{singleIdempotent}. \end{proof}

This corollary will be used to prove the next theorem, which is the key algebraic observation underlying our main result (Theorem \ref{subalgebraOfBrauer}).

\begin{theorem} \label{quotientingPlus} Let $A$ be an $R$-algebra, let $M$ be a right $A$-module and let $N$ be a left $A$-module. Let $I$ be a two-sided ideal of $A$ which acts trivially on $M$ and $N$, and which as a left ideal is a direct sum $I \cong J_1 \oplus \dots \oplus J_k$. Suppose that each $J_i$ is generated as a left ideal by finitely many commuting idempotents. Then $$\Tor^A_*(M,N) \cong \Tor^{\faktor{A}{I}}_*(M,N).$$ \end{theorem}

Before proving this theorem, we note that it implies the following corollary, by taking a chain of ideals $0 = I_{-1} \leq I_0 \leq I_1 \leq \dots \leq I_m \leq A,$ and $\ell \geq 0$, and supposing that for each $i \geq \ell$, the hypothesis of that theorem holds for the image of each $I_i$ inside $\faktor{A}{I_{i-1}}$.

\begin{corollary} \label{quotientingInductively} Let $A$ be an $R$-algebra, let $M$ be a right $A$-module and let $N$ be a left $A$-module. Let $0 \leq \ell \leq m$. Suppose that we have a chain of two-sided ideals $$0=I_{-1} \leq I_0 \leq I_1 \leq \dots \leq I_m \leq A,$$ where each $I_i$ acts trivially on $M$ and $N$, and for $i \geq \ell$, as left $A$-modules we have an isomorphism $$\faktor{I_i}{I_{i-1}} \cong \faktor{J_{i,1}}{I_{i-1}} \oplus \dots \oplus \faktor{J_{i,k_i}}{I_{i-1}},$$ for left ideals $J_{i,j}$ generated by finitely many commuting idempotents (the sum being direct means equivalently that $J_{i,j} \cap J_{i,j'} \subset I_{i-1}$ whenever $j \neq j'$). Then \[\pushQED{\qed} 
\Tor^{\faktor{A}{I_{\ell-1}}}_*(M,N) \cong \Tor^{\faktor{A}{I_\ell}}_*(M,N) \cong \dots \cong \Tor^{\faktor{A}{I_m}}_*(M,N). \qedhere
\popQED
\]\end{corollary}

This corollary is the main result of this section: it is the algebra underlying our main result, Theorem \ref{subalgebraOfBrauer}. The artificial asymmetry between the roles played by $\ell$ and $m$ (that the chain of ideals $I_i$ is assumed to be defined even when $i \leq \ell$) becomes more natural in that theorem.

\begin{remark} Theorem \ref{quotientingPlus} actually implies a bit more than Corollary \ref{quotientingInductively}: the generators of the $J_{i,j}$ need only commute and be idempotent \emph{modulo $I_{i-1}$}, but we will not need this extra generality. \end{remark}

\begin{proof}[Proof of Theorem \ref{quotientingPlus}] We will show that $\Tor^A_*(M,N)$ and $\Tor^{\faktor{A}{I}}_*(M,N)$ are the homology of the same chain complex.

Let $P_*$ be a free resolution (by right $A$-modules) of $M$ over $A$, so that $\Tor^A_*(M,N)$ is the homology of $P_* \otimes_A N$. We have
\begin{align*}
    P_* \otimes_A N & \cong P_* \otimes_A A \otimes_A N \\
    & \cong (P_* \otimes_A \faktor{A}{I}) \otimes_{\faktor{A}{I}} N,
\end{align*} since $I$ acts trivially on $N$. Since $P_*$ consists of free $A$-modules, each $$P_i \otimes_A \faktor{A}{I} \cong (\bigoplus A) \otimes_A \faktor{A}{I} \cong \bigoplus \faktor{A}{I}$$ is a free (right) $\faktor{A}{I}$-module. It therefore suffices to show that $P_* \otimes_A \faktor{A}{I}$ is a resolution of $M$, which is to say we must show that the homology of this complex, which in degree $q$ is $\Tor_q^A(M,\faktor{A}{I})$, is $M$ when $q=0$ and $0$ otherwise.

By Corollary \ref{idempotency}, $I \cong J_1 \oplus \dots \oplus J_k$ is a direct sum of projective $A$-modules, hence also projective. The exact sequence $A \xleftarrow{} I \xleftarrow{} 0$ is therefore a (very short) projective resolution of $\faktor{A}{I}$ over $A$, so $\Tor_*^A(M,\faktor{A}{I})$ is the homology of $M \otimes A \xleftarrow{} M \otimes I \xleftarrow{} 0$. Now, $I$ is assumed to act trivially on $M$, and the direct sum decomposition of $I$ implies that $M \otimes I$ is generated by elements of the form $m \otimes a e$, for $e$ one of the idempotent generators of some $J_i$, $a \in A$, and $m \in M$. Then we may write $$m \otimes a e = m \otimes a e^2 = (m \cdot ae ) \otimes e = 0,$$ since $ae \in I$ is assumed to act trivially on $M$. Since $M \otimes I$ is generated by such elements, we have $M \otimes I \cong 0$, and the result follows. \end{proof}

\section{Applications: The rook algebras} \label{sectionRook}

Recall from Definition \ref{defRook} that the rook algebra $\rook_n(\varepsilon)$ is the subalgebra of $\rook \Brauer_n(\delta, \varepsilon)$ on those diagrams having no left-to-left or right-to-right connections. In this section we will give our first application, Theorem \ref{rookInvertible}: when $\varepsilon$ is invertible, the homology of the rook algebras coincides with that of the symmetric groups. This is the simplest application of our framework, essentially because the idempotents we will identify actually commute in the algebra itself, without taking successive quotients. In particular, the case $k=0$ of Theorem \ref{quotientingPlus}, which we proved in the last section, will suffice to play the role that Theorem \ref{subalgebraOfBrauer} will play in later applications. The statement of this special case is as follows.

\begin{theorem} \label{quotienting} Let $A$ be an $R$-algebra, let $M$ be a right $A$-module and let $N$ be a left $A$-module. Let $I$ be a two-sided ideal of $A$ which acts trivially on $M$ and $N$, and which as a left ideal is generated by finitely many commuting idempotents. Then \[\pushQED{\qed} 
\Tor^A_*(M,N) \cong \Tor^{\faktor{A}{I}}_*(M,N). \qedhere
\popQED
\] \end{theorem}

For us, $I$ will be the ideal $\rook_n(\varepsilon) \cap I_{n-1}$ of rook diagrams having fewer than $n$ left-to-right connections. To apply the theorem, we first have to find some idempotents.

Let $\rho_i$ be the element of $\rook_n(\varepsilon)$ obtained from the identity diagram by deleting the $i$-th strand, so that for example $\rho_2$ in $\rook_5$ is the diagram
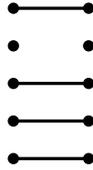
\begin{center}
\begin{tikzpicture}
\fill (0,0)           circle (.75mm) node[left=2pt](a5) {};
\fill (0,\nodeheight) circle (.75mm) node[left=2pt](a4) {};
\multiply\nodeheight by 2
\fill (0,\nodeheight) circle (.75mm) node[left=2pt](a3) {};
\divide\nodeheight by 2
\multiply\nodeheight by 3
\fill (0,\nodeheight) circle (.75mm) node[left=2pt](a2) {};
\divide\nodeheight by 3
\multiply\nodeheight by 4
\fill (0,\nodeheight) circle (.75mm) node[left=2pt](a1) {};
\divide\nodeheight by 4

\fill (\nodewidth,0) circle (.75mm) node[right=2pt](b5) {};
\fill (\nodewidth,\nodeheight) circle (.75mm) node[right=2pt](b4) {};
\multiply\nodeheight by 2
\fill (\nodewidth,\nodeheight) circle (.75mm) node[right=2pt](b3) {};
\divide\nodeheight by 2
\multiply\nodeheight by 3
\fill (\nodewidth,\nodeheight)  circle (.75mm) node[right=2pt](b2) {};
\divide\nodeheight by 3
\multiply\nodeheight by 4
\fill (\nodewidth,\nodeheight)  circle (.75mm) node[right=2pt](b1) {};
\divide\nodeheight by 4

\draw[e] (a1) to[out=0, in=180] (b1);
\draw[e] (a2) to[out=0, in=180] (b2);
\draw[e] (a3) to[out=0, in=180] (b3);
\draw[e] (a5) to[out=0, in=180] (b5); \end{tikzpicture}
\end{center}

For each $i$ we have the identity $\rho_i^2= \varepsilon \rho_i$, and hence the following lemma.

\begin{lemma} If $\varepsilon \in R$ is invertible, then $\varepsilon^{-1} \rho_i$ is idempotent. \qed \label{invertibleVarepsilon}
\end{lemma}

\begin{lemma} \label{myFirstIdeal} Regardless of whether $\varepsilon \in R$ is invertible, the left ideal $\rook_n(\varepsilon) \cdot \rho_i < \rook_n(\varepsilon)$ has $R$-basis consisting of those diagrams where the $i$-th node on the right is not connected to anything on the left. \end{lemma}

\begin{proof} Let $y$ be a diagram where the $i$-th node on the right is not connected to anything on the left. There must be at least one vacant node, say $j$, on the left of $y$. Let $y'$ be the diagram obtained by adding a left-to-right connection from $i$ to $j$. We then have $y' \rho_i = y$, so $y$ is contained in $\rook_n(\varepsilon) \cdot \rho_i$. Since the latter is closed under $R$-linear combinations, the result follows. \end{proof}

We are now ready for our first application.

\begin{theorem} \label{rookInvertible} Let $\varepsilon \in R$ be invertible, and let $\rook_n = \rook_n(\varepsilon)$. Let $\Sigma_n$ be the symmetric group, and let $R \Sigma_n$ be the group algebra. Then $$\Tor_*^{\rook_n}(\t,\t) \cong \Tor_*^{R \Sigma_n}(\t,\t) =: H_*(\Sigma_n;\t).$$ \end{theorem}

\begin{proof} Let $I= \rook_n \rho_1 + \dots + \rook_n \rho_n$ be the left ideal generated by the $\rho_i$. This ideal is equivalently generated by the elements $\varepsilon^{-1} \rho_i$. It is directly verified that these elements commute, and they are idempotent by Lemma \ref{invertibleVarepsilon}. By definition $I$ acts trivially on $\t$, so we may apply Theorem \ref{quotienting} with $M=N=\t$, and it suffices to identify the quotient $\faktor{\rook_n}{I}$.

A diagram having at least one missing left-to-right connection lies in some $\rook_n \rho_i \subset I$ by Lemma \ref{myFirstIdeal}. It follows that $I$ has $R$-basis consisting of the diagrams with at least one missing left-to-right connection. The quotient $\faktor{\rook_n}{I}$ then has $R$-basis consisting of the diagrams with no missing left-to-right connections, which is to say (Proposition \ref{retract}) that $\faktor{\rook_n}{I} \cong R \Sigma_n$. Since this isomorphism carries the trivial module for $\rook_n$ to the trivial module for the symmetric group, this identifies the right hand side of the isomorphism of Theorem \ref{quotienting} with the homology of the symmetric group, completing the proof. \end{proof}

\section{Specialisation to diagram algebras} \label{SectionDiagram}

This section is devoted to the proof of the main result of this paper, Theorem \ref{subalgebraOfBrauer}, which may be thought of as a specialisation of Corollary \ref{quotientingInductively} to diagram algebras.

Recall that $I_i$ denotes the two-sided ideal of $\rook \Brauer_n(\delta, \varepsilon)$ with $R$-basis consisting of the diagrams having at most $i$ left-to-right connections. Recall also that $P_i$ denotes the set of link states with precisely $i$ defects, and that for $p \in P_i$, $J_p$ denotes the left ideal of $\rook \Brauer_n(\delta, \varepsilon)$ which is the $R$-span of diagrams having right link state obtained from $p$ by a (perhaps empty) sequence of {\mOne}s and {\mTwo}s.

\begin{proof}[Proof of Theorem \ref{subalgebraOfBrauer}] We are given $\ell, m$ with $0 \leq \ell \leq m \leq n-1$, and $A$, a subalgebra of $\rook \Brauer_n(\delta, \varepsilon)$, such that \begin{itemize}
    \item as an $R$-module, $A$ is free on a subset of the rook-Brauer $n$-diagrams, and
    \item for $i$ in the range $\ell \leq i \leq m$, whenever $A$ contains at least one diagram with some right link state $p \in P_i$, then $A$ contains an idempotent $e_p$ such that in $A$ we have equality of left ideals $A \cdot e_p = A \cap J_p$.
\end{itemize} We must deduce a chain of isomorphisms $$\Tor^{\faktor{A}{A \cap I_{\ell-1}}}_*(\t,\t) \cong \Tor^{\faktor{A}{A \cap I_{\ell}}}_*(\t,\t) \cong \dots \cong \Tor^{\faktor{A}{A \cap I_{m}}}_*(\t,\t).$$

We wish to apply Corollary \ref{quotientingInductively} to the chain of two-sided ideals $$0 = A \cap I_{-1} \leq A \cap I_0 \leq A \cap I_1 \leq \dots \leq A \cap I_{m} \leq A,$$ where $I_i$ is the two-sided ideal of $\rook \Brauer_n(\delta, \varepsilon)$ with $R$-basis diagrams having at most $i$ left-to-right connections. Since $A$ is assumed to be generated as an $R$-module by a subset of the rook-Brauer $n$-diagrams, $A \cap I_i$ admits the same description: it is free as an $R$-module on those diagrams \emph{in $A$} which have at most $i$ left-to-right connections.

Let $i$ lie in the range $\ell \leq i \leq m$. First note that $I_i$ acts trivially on $\t$ for $i \leq n-1$, so $A \cap I_i$ also does. This is the reason for the requirement $m \leq n-1$. Note that since we assume $\ell \geq 0$, it makes sense to speak about $I_{\ell-1}$ (under the convention that $I_{-1}=0$).

The quotient $\faktor{I_i}{I_{i-1}}$ has $R$-basis given by the diagrams with precisely $i$ left-to-right connections. Since every such diagram has \emph{some} right link state, we get the equality $$\faktor{I_i}{I_{i-1}} = \sum_{p \in P_i} \faktor{J_p}{I_{i-1}}.$$ Parenthetically, we note that it is only necessary to work modulo $I_{i-1}$ here because not every diagram having at most $i$ left-to-right connections is actually obtained from one having precisely $i$ left-to-right connections by a sequence of {\mOne}s and {\mTwo}s (this is for the silly reason that a {\mOne} removes \emph{two} defects, thereby `skipping' from $i$ to $i-2$). Other solutions to this technical irritation are surely possible, but we will need to work modulo $I_{i-1}$ later anyway. Thus, since $A$, $I_i$, $I_{i-1}$ and each $J_p$ are free on a subset of the rook-Brauer diagrams, $$\faktor{A \cap I_i}{A \cap I_{i-1}} = \sum_{p \in P_i} \faktor{A \cap J_p}{A \cap I_{i-1}}.$$ We assumed that each $A \cap J_p$ is a principal left ideal generated by some idempotent $e_p$, so in particular each $A \cap J_p$ is generated by a family of commuting idempotents.

It remains only to show that the above sum is actually direct, that is, that for $p \neq q \in P_i$, the intersection of $A \cap J_p$ and $A \cap J_q$ is contained in $A \cap I_{i-1}$. It suffices to show that $J_p \cap J_q \subset I_{i-1}$: at this stage working modulo $I_{i-1}$ really is materially important. To see this, recall that $J_p$ is free as an $R$-module on diagrams having right link state obtained from $p$ by a sequence of {\mOne}s and {\mTwo}s. $J_p \cap J_q$ is therefore generated by diagrams with right link states that can be obtained from \emph{both} $p$ and $q$ by such a sequence, and if $p$ and $q$ are distinct with $i$ defects each of these must involve at least one {\mOne} or {\mTwo}, which reduces the number of defects, yielding a diagram in $I_{i-1}$. This completes the proof. \end{proof}

\section{Applications: Invertible parameters} \label{applicationsInvertible}

The remainder of this paper treats applications of Theorem \ref{subalgebraOfBrauer}, roughly in order of increasing complexity. In this section we give the `invertible-parameter' applications.

These applications will be powered by the following lemma. Recall that the \emph{Brauer algebra} $\Brauer_n(\delta)$ is the subalgebra of $\RB_n(\delta,\varepsilon)$ on diagrams having no missing edges.

\begin{lemma} \label{mirrorDiagram} Let $p \in P_i$ be a link state with no missing edges. Let $d_p$ be the diagram whose left and right link states are both $p$, and all of whose edges are horizontal. If $y$ lies in $\Brauer_n(\delta) \cap J_p$, then $$y d_p = \delta^{\frac{1}{2}(n-i)} y.$$ \end{lemma}

In other words, $\Brauer_n(\delta) \cap J_p$ is contained in the eigenspace of $d_p$ having eigenvalue $\delta^{\frac{1}{2}(n-i)}$. We are reassured that $\frac{1}{2}(n-i)$ is an integer by noticing that it is precisely the number of connections in $p$ (we use here the fact that $p$ has no missing edges).

\begin{proof} We give an illustrative example (Example \ref{mirrorDiagramEx}) after the proof. The reader may find it helpful to `carry this example along'.

First, since the multiplication is bilinear, it suffices to prove the lemma when $y$ is a diagram.

We will argue that the underlying diagram of $y d_p$ has all connections from $y$, and differs from $y$ only in the possible appearance of some factors of $\delta$. In the process, we will see that the power of $\delta$ which appears must be the number of connections in $p$.

First consider the form $y$ must take. Since $y$ lies in $\Brauer_n(\delta)$, it must have no missing edges. This means that any sequence of {\mOne}s and {\mTwo}s which witnesses the fact that $y$ lies in $J_p$ must consist entirely of {\mOne}s. In other words, the right link state of $y$ is obtained from $p$ by a sequence of {\mOne}s.

Since $d_p$ has horizontal connections everywhere apart from at the connections of $p$, it immediately follows that each connection in $p$ gives a copy of $\delta$, and every other connection of $y$ is preserved, as required. This completes the proof. \end{proof}

\begin{example} \label{mirrorDiagramEx} For illustration, we give an example of the equation $$y d_p = \delta^{\frac{1}{2}(n-i)} y.$$

We let \begin{center}
\[
p = 
\quad
\begin{tikzpicture}[x=1.5cm,y=-.5cm,baseline=-1.05cm]

\node[v] (b1) at (1,0) {};
\node[v] (b2) at (1,1) {};
\node[v] (b3) at (1,2) {};
\node[v] (b4) at (1,3) {};
\node[v] (b5) at (1,4) {};

\draw[e] (0,2) to[out=0, in=180] (b3);
\draw[e] (0,1) to[out=0, in=180] (b2);
\draw[e] (0,4) to[out=0, in=180] (b5);

\draw[e] (b1) to[out=180, in=180] (b4);

\end{tikzpicture}
\quad
, \textrm{ so that }
d_p =
\quad
\begin{tikzpicture}[x=1.5cm,y=-.5cm,baseline=-1.05cm]

\node[v] (b1) at (0,0) {};
\node[v] (b2) at (0,1) {};
\node[v] (b3) at (0,2) {};
\node[v] (b4) at (0,3) {};
\node[v] (b5) at (0,4) {};

\node[v] (c1) at (1,0) {};
\node[v] (c2) at (1,1) {};
\node[v] (c3) at (1,2) {};
\node[v] (c4) at (1,3) {};
\node[v] (c5) at (1,4) {};

\draw[e] (b3) to[out=0, in=180] (c3);
\draw[e] (b5) to[out=0, in=180] (c5);
\draw[e] (b2) to[out=0, in=180] (c2);

\draw[e] (b1) to[out=0, in=0] (b4);

\draw[e] (c1) to[out=180, in=180] (c4);
\end{tikzpicture}
\quad
,
\]
\end{center}
and suppose that we are given
\begin{center}
\[
y =
\begin{tikzpicture}[x=1.5cm,y=-.5cm,baseline=-1.05cm]
\node[v] (a1) at (0,0) {};
\node[v] (a2) at (0,1) {};
\node[v] (a3) at (0,2) {};
\node[v] (a4) at (0,3) {};
\node[v] (a5) at (0,4) {};

\node[v] (b1) at (1,0) {};
\node[v] (b2) at (1,1) {};
\node[v] (b3) at (1,2) {};
\node[v] (b4) at (1,3) {};
\node[v] (b5) at (1,4) {};

\draw[e] (a1) to[out=0, in=180] (b2);

\draw[e] (a2) to[out=0, in=0] (a3);
\draw[e] (a4) to[out=0, in=0] (a5);

\draw[e] (b1) to[out=180, in=180] (b4);
\draw[e] (b3) to[out=180, in=180] (b5);

\end{tikzpicture}
\quad
\in \Brauer_n(\delta) \cap J_p.
\]
\end{center}

Note that the right link state of $y$ is obtained from $p$ by only a single {\mOne}. Note also that $\frac{1}{2}(n-i)=\frac{1}{2}(5-3)=1$. We can now compute:
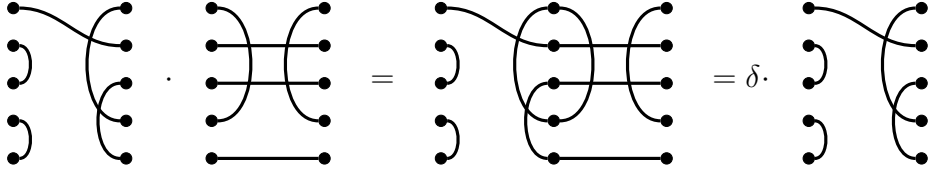
\begin{center}
\[
\begin{tikzpicture}[x=1.5cm,y=-.5cm,baseline=-1.05cm]
\node[v] (a1) at (0,0) {};
\node[v] (a2) at (0,1) {};
\node[v] (a3) at (0,2) {};
\node[v] (a4) at (0,3) {};
\node[v] (a5) at (0,4) {};

\node[v] (b1) at (1,0) {};
\node[v] (b2) at (1,1) {};
\node[v] (b3) at (1,2) {};
\node[v] (b4) at (1,3) {};
\node[v] (b5) at (1,4) {};

\draw[e] (a1) to[out=0, in=180] (b2);

\draw[e] (a2) to[out=0, in=0] (a3);
\draw[e] (a4) to[out=0, in=0] (a5);

\draw[e] (b1) to[out=180, in=180] (b4);
\draw[e] (b3) to[out=180, in=180] (b5);

\end{tikzpicture}
\quad
\cdot
\quad
\begin{tikzpicture}[x=1.5cm,y=-.5cm,baseline=-1.05cm]

\node[v] (b1) at (0,0) {};
\node[v] (b2) at (0,1) {};
\node[v] (b3) at (0,2) {};
\node[v] (b4) at (0,3) {};
\node[v] (b5) at (0,4) {};

\node[v] (c1) at (1,0) {};
\node[v] (c2) at (1,1) {};
\node[v] (c3) at (1,2) {};
\node[v] (c4) at (1,3) {};
\node[v] (c5) at (1,4) {};

\draw[e] (b3) to[out=0, in=180] (c3);
\draw[e] (b5) to[out=0, in=180] (c5);
\draw[e] (b2) to[out=0, in=180] (c2);

\draw[e] (b1) to[out=0, in=0] (b4);

\draw[e] (c1) to[out=180, in=180] (c4);
\end{tikzpicture}
\quad
= 
\quad
\begin{tikzpicture}[x=1.5cm,y=-.5cm,baseline=-1.05cm]

\node[v] (a1) at (0,0) {};
\node[v] (a2) at (0,1) {};
\node[v] (a3) at (0,2) {};
\node[v] (a4) at (0,3) {};
\node[v] (a5) at (0,4) {};

\node[v] (b1) at (1,0) {};
\node[v] (b2) at (1,1) {};
\node[v] (b3) at (1,2) {};
\node[v] (b4) at (1,3) {};
\node[v] (b5) at (1,4) {};

\node[v] (c1) at (2,0) {};
\node[v] (c2) at (2,1) {};
\node[v] (c3) at (2,2) {};
\node[v] (c4) at (2,3) {};
\node[v] (c5) at (2,4) {};

\draw[e] (a1) to[out=0, in=180] (b2);

\draw[e] (a2) to[out=0, in=0] (a3);
\draw[e] (a4) to[out=0, in=0] (a5);

\draw[e] (b1) to[out=180, in=180] (b4);
\draw[e] (b3) to[out=180, in=180] (b5);

\draw[e] (b3) to[out=0, in=180] (c3);
\draw[e] (b5) to[out=0, in=180] (c5);
\draw[e] (b2) to[out=0, in=180] (c2);

\draw[e] (b1) to[out=0, in=0] (b4);

\draw[e] (c1) to[out=180, in=180] (c4);

\end{tikzpicture}
\quad
= \delta \cdot
\quad
\begin{tikzpicture}[x=1.5cm,y=-.5cm,baseline=-1.05cm]

\node[v] (a1) at (0,0) {};
\node[v] (a2) at (0,1) {};
\node[v] (a3) at (0,2) {};
\node[v] (a4) at (0,3) {};
\node[v] (a5) at (0,4) {};

\node[v] (c1) at (1,0) {};
\node[v] (c2) at (1,1) {};
\node[v] (c3) at (1,2) {};
\node[v] (c4) at (1,3) {};
\node[v] (c5) at (1,4) {};

\draw[e] (a1) to[out=0, in=180] (b2);

\draw[e] (a2) to[out=0, in=0] (a3);
\draw[e] (a4) to[out=0, in=0] (a5);

\draw[e] (c3) to[out=180, in=180] (c5);
\draw[e] (c1) to[out=180, in=180] (c4);

\end{tikzpicture}
\quad
\]
\end{center}
\end{example}

\subsection{The (rook-)Brauer algebras} \label{subsectionInvBrauer}

Recall from Definition \ref{defBrauer} that the \emph{Brauer algebra} $\Brauer_n(\delta)$ is the subalgebra of $\RB_n(\delta,\varepsilon)$ on diagrams having no missing edges.

Consider first the full rook-Brauer algebra $\RB_n(\delta,\varepsilon)$. First, if $\varepsilon$ is invertible, notice that the same analysis as for the rook algebra applies: we still have the elements $\rho_i$ of Section \ref{sectionRook}, and applying Theorem \ref{quotienting} to the two-sided ideal they generate immediately gives the following theorem.

\begin{theorem} \label{rookBrauerInvertible} If $\varepsilon \in R$ is invertible, then (for any $\delta \in R$) we get \[\pushQED{\qed} 
\Tor_*^{\RB_n(\delta, \varepsilon)}(\t,\t) \cong \Tor_*^{\Brauer_n(\delta)}(\t,\t).\qedhere
\popQED
\] \end{theorem}

Boyd, Hepworth, and Patzt \cite{BHP} have proven homological stability for $\Brauer_n(\delta)$. Combining their result with Theorem \ref{rookBrauerInvertible} gives the following corollary.

\begin{corollary} \label{funOne} If $\varepsilon \in R$ is invertible, then, for any $\delta \in R$, the stabilisation map 
$$\Tor_q^{\RB_{n-1}(\delta, \varepsilon)}(\t,\t) \to \Tor_q^{\RB_{n}(\delta, \varepsilon)}(\t,\t)$$ is an isomorphism for $n \geq 2q+1$. \qed \end{corollary}

We can also use Theorem \ref{subalgebraOfBrauer} to recover the following result of Boyd, Hepworth, and Patzt.

\begin{theorem}[\cite{BHP}] \label{BrauerRecovery} If $\delta \in R$ is invertible, then
$$ \Tor_*^{\Brauer_n(\delta)}(\t,\t) \cong \Tor_*^{R \Sigma_n}(\t,\t).$$ \end{theorem}

\begin{proof} Since $\Brauer_n = \Brauer_n(\delta)$ is free as an $R$-module on a subset of the rook-Brauer diagrams, we may attempt to apply Theorem \ref{subalgebraOfBrauer} with $\ell =0$ and $m = n-1$. This will give the correct conclusion, since $\faktor{\Brauer_n}{\Brauer_n \cap I_{n-1}} \cong R \Sigma_n$.

Let $i$ lie in the range $0 \leq i \leq n-1$. Let $p \in P_i$ be a link state with $i$ defects, that actually occurs as the right link state of some diagram in $\Brauer_n$. This means precisely that $p$ has no missing edges. We must construct an idempotent $e_p $ which generates $\Brauer_n \cap J_p$ as a left ideal.

Let $d_p$ be the diagram whose left and right link states are both $p$, and all of whose edges are horizontal. Lemma \ref{mirrorDiagram} (with $y=d_p$) gives $(d_p)^2 = \delta^{\frac{1}{2}(n-i)} d_p$, so $e_p = \delta^{-\frac{1}{2}(n-i)} d_p$ is an idempotent.

We must establish the equality of left ideals $\Brauer_n \cdot e_p = \Brauer_n \cap J_p$. The inclusion $\Brauer_n \cdot e_p  \subset \Brauer_n \cap J_p$ follows immediately, since $e_p$ lies in $J_p$ by construction. To see the reverse inclusion, let $y \in \Brauer_n \cap J_p$. Lemma \ref{mirrorDiagram} gives that $y d_p = \delta^{\frac{1}{2}(n-i)} y$, so $y e_p = y$, and so $y$ lies in $\Brauer_n  \cdot e_p$.

Thus, $\Brauer_n \cap J_p = \Brauer_n \cdot e_p$, and the result follows. \end{proof}

\subsection{The Temperley-Lieb algebras} \label{subsectionInvTL}

Recall that the \emph{Temperley-Lieb algebra} $\TL_n(\delta)$ is the subalgebra of $\Brauer_n(\delta) \subset \RB_n(\delta,\varepsilon)$ on planar diagrams (having no missing edges).

First, note that the analogue of Theorem \ref{rookBrauerInvertible} does not hold here (at least, it does not hold for the same reason). Namely, letting the \emph{rook-Temperley-Lieb} algebra $\rook \TL_n(\delta, \varepsilon)$ be the subalgebra of $\rook \Brauer_n(\delta,\varepsilon)$ on the planar diagrams (i.e. we now allow missing edges), we see that not every rook-Temperley-Lieb diagram can be `extended' to a Temperley-Lieb diagram by adding edges, since this might violate planarity. Here is one of the two simplest possible examples of such a diagram, which occurs in $\rook \TL_2$:

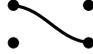
\begin{center}
\begin{tikzpicture}
\fill (0,0)           circle (.75mm) node[left=2pt](a5){};
\fill (0,\nodeheight) circle (.75mm) node[left=2pt](a4){};

\fill (\nodewidth,0) circle (.75mm) node[right=2pt](b5){};
\fill (\nodewidth,\nodeheight) circle (.75mm) node[right=2pt](b4){};

\draw[e] (a5) to[out=0, in=180] (b4); \end{tikzpicture}
\end{center}

This means that the rook-Temperley-Lieb algebra is not generated by its Temperley-Lieb subalgebra plus the elements $\rho_i$, so the key observation that powered the proof of Theorem \ref{rookBrauerInvertible} fails in this context.

On the other hand, we are able to recover the following result of Boyd and Hepworth, which is the analogue of Theorem \ref{BrauerRecovery}.

\begin{theorem}[\cite{BH}] If $\delta \in R$ is invertible, then $\Tor_q^{\TL_n(\delta)}(\t,\t) = 0$ for $q>0$. \label{TLRecovery} \end{theorem}

What follows is essentially the same as the proof of Theorem \ref{BrauerRecovery}, except that we have to say everything a little more carefully to ensure planarity.

\begin{proof} Since $\TL_n = \TL_n(\delta)$ is free as an $R$-module on a subset of the rook-Brauer diagrams, we may attempt to apply Theorem \ref{subalgebraOfBrauer}, with $\ell = 0$ and $m = n-1$. This will give the correct conclusion: the only $\TL_n$-diagram having $n$ left-to-right connections is the identity, so $\faktor{\TL_n}{\TL_n \cap I_{n-1}} \cong R$.

Let $p \in P_i$ be a link state with $i$ defects, occurring as the right link state of some diagram in $\TL_n$. This means that $p$ has no missing edges and that none of its defects lie inside the arc of any of its connections. We must construct an idempotent $e_p $ which generates $\TL_n \cap J_p$ as a left ideal.

Take $d_p$ to be the diagram whose left and right link states are both $p$, with all edges horizontal. Since no defects of $p$ lie inside the arc of any of its connections, $d_p$ is planar, and gives a diagram in $\TL_n$.

By Lemma \ref{mirrorDiagram}, $(d_p)^2 = \delta^{\frac{1}{2}(n-i)} d_p$, so setting $e_p = \delta^{-\frac{1}{2}(n-i)} d_p$ gives an idempotent. We must now show that $\TL_n \cdot e_p = \TL_n \cap J_p$. The inclusion $\TL_n \cdot e_p \subset \TL_n \cap J_p$ follows immediately, since $e_p$ lies in $J_p$ by construction, and the reverse inclusion again follows by applying Lemma \ref{mirrorDiagram} to an arbitrary $y \in \TL_n \cap J_p$. It is important here that in the equation $y d_p = \delta^{\frac{1}{2}(n-i)} y$, the $y$ appearing on the left is planar (though this is somewhat tautological) because this shows that $y$ lies in the left $\TL_n$-span (rather than just the $\Brauer_n$-span) of $d_p$, hence of $e_p$.

We are therefore done by Theorem \ref{subalgebraOfBrauer}. \end{proof}

\section{Applications: Sroka-type theorems} \label{applicationsSroka}

In this section, we prove Theorem \ref{generalisedSroka} and Theorem \ref{generalisedBrauerSroka}. The basic pattern is as in the last section, but we now use Lemma \ref{LSControl} to verify the hypotheses of Theorem \ref{subalgebraOfBrauer}, in place of Lemma \ref{mirrorDiagram}. The extra complexity in this lemma is because we are no longer assuming that $\delta$ is invertible - we must now avoid producing factors of $\delta$, because we can no longer just rescale to get rid of them.

\subsection{Double and sesqui- diagrams}

Given two diagrams $x$ and $y$, the first step in computing the product $x \cdot y$ is to concatenate the two diagrams. We introduce the following ad hoc language, which will be useful for discussing these `concatenated, but not yet resolved' diagrams, and will help in making some of the proofs rigorous.

\begin{definition} Given diagrams $x$ and $y$, the \emph{double diagram} $(x,y)$ consists of three columns of $n$ nodes, having $x$ as its left half, and $y$ as its right half. In other words, we think of $x$ and $y$ as graphs, and identify the right hand column of nodes in $x$ with the left hand column of nodes in $y$. 

We will write, counting from top to bottom, $l_1, \dots l_n$ for the nodes on the left, $m_1, \dots m_n$ for the nodes in the middle, and $r_1, \dots r_n$ for the nodes on the right. Intrinsically, we think of the double diagram as a choice of a set of edges connecting this fixed set of nodes. For nodes $$u,v \in \{ l_1, \dots l_n , m_1, \dots m_n, r_1, \dots r_n\},$$ we write $$u \underset{x,y}{\sim} v$$ if there is a sequence of edges in the double diagram $(x,y)$ connecting $u$ and $v$. The relation $\underset{x,y}{\sim}$ is an equivalence relation on the set of symbols $\{ l_1, \dots l_n , m_1, \dots m_n, r_1, \dots, r_n\}$, and this equivalence relation determines the algebra product $xy$. \end{definition}

\begin{example} In $\Brauer_7$, if 

\begin{center}
$x =$
\quad
\begin{tikzpicture}[x=1.5cm,y=-.5cm,baseline=-1.05cm]

\def\wid{2}

\node[v] (a1) at (0,0) {};
\node[v] (a2) at (0,1) {};
\node[v] (a3) at (0,2) {};
\node[v] (a4) at (0,3) {};
\node[v] (a5) at (0,4) {};
\node[v] (a6) at (0,5) {};
\node[v] (a7) at (0,6) {};

\node[v] (b1) at (1* \wid,0) {};
\node[v] (b2) at (1* \wid,1) {};
\node[v] (b3) at (1* \wid,2) {};
\node[v] (b4) at (1* \wid,3) {};
\node[v] (b5) at (1* \wid,4) {};
\node[v] (b6) at (1* \wid,5) {};
\node[v] (b7) at (1* \wid,6) {};

\draw[e] (a4) to[out=0, in=180] (b6);
\draw[e] (a1) to[out=0, in=180] (b5);
\draw[e] (a2) to[out=0, in=180] (b1);

\draw[e] (a3) to[out=0, in=0] (a7);
\draw[e] (a5) to[out=0, in=0] (a6);

\draw[e] (b2) to[out=180, in=180] (b3);
\draw[e] (b2) to[out=180, in=180] (b3);
\draw[e] (b4) to[out=180, in=180] (b7);

\end{tikzpicture}
\quad
and $y=$
\quad
\begin{tikzpicture}[x=1.5cm,y=-.5cm,baseline=-1.05cm]

\def\wid{2}

\node[v] (b1) at (1* \wid,0) {};
\node[v] (b2) at (1* \wid,1) {};
\node[v] (b3) at (1* \wid,2) {};
\node[v] (b4) at (1* \wid,3) {};
\node[v] (b5) at (1* \wid,4) {};
\node[v] (b6) at (1* \wid,5) {};
\node[v] (b7) at (1* \wid,6) {};

\node[v] (c1) at (2* \wid,0) {};
\node[v] (c2) at (2* \wid,1) {};
\node[v] (c3) at (2* \wid,2) {};
\node[v] (c4) at (2* \wid,3) {};
\node[v] (c5) at (2* \wid,4) {};
\node[v] (c6) at (2* \wid,5) {};
\node[v] (c7) at (2* \wid,6) {};

\draw[e] (c1) to[out=180, in=180] (c5);
\draw[e] (c4) to[out=180, in=180] (c7);

\draw[e] (b1) to[out=0, in=0] (b4);
\draw[e] (b3) to[out=0, in=0] (b2);

\draw[e] (b5) to[out=0, in=180] (c3);
\draw[e] (b7) to[out=0, in=180] (c6);
\draw[e] (b6) to[out=0, in=180] (c2);

\end{tikzpicture}
\end{center}

then the double diagram formed from $x$ and $y$ is

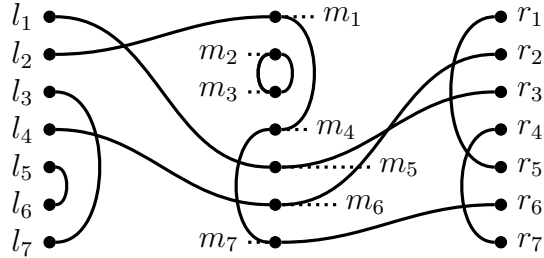
\begin{center}
\quad
\begin{tikzpicture}[x=1.5cm,y=-.5cm,baseline=-1.05cm]

\def\wid{2}

\fill (0,0) circle (.75mm) node[left=2pt](a) {$l_1$};
\fill (0,1) circle (.75mm) node[left=2pt](a) {$l_2$};
\fill (0,2) circle (.75mm) node[left=2pt](a) {$l_3$};
\fill (0,3) circle (.75mm) node[left=2pt](a) {$l_4$};
\fill (0,4) circle (.75mm) node[left=2pt](a) {$l_5$};
\fill (0,5) circle (.75mm) node[left=2pt](a) {$l_6$};
\fill (0,6) circle (.75mm) node[left=2pt](a) {$l_7$};

\fill (2*\wid,0) circle (.75mm) node[right=2pt](a) {$r_1$};
\fill (2*\wid,1) circle (.75mm) node[right=2pt](a) {$r_2$};
\fill (2*\wid,2) circle (.75mm) node[right=2pt](a) {$r_3$};
\fill (2*\wid,3) circle (.75mm) node[right=2pt](a) {$r_4$};
\fill (2*\wid,4) circle (.75mm) node[right=2pt](a) {$r_5$};
\fill (2*\wid,5) circle (.75mm) node[right=2pt](a) {$r_6$};
\fill (2*\wid,6) circle (.75mm) node[right=2pt](a) {$r_7$};

\node[v] (a1) at (0,0) {};
\node[v] (a2) at (0,1) {};
\node[v] (a3) at (0,2) {};
\node[v] (a4) at (0,3) {};
\node[v] (a5) at (0,4) {};
\node[v] (a6) at (0,5) {};
\node[v] (a7) at (0,6) {};

\node[v] (b1) at (1* \wid,0) {};
\node[v] (b2) at (1* \wid,1) {};
\node[v] (b3) at (1* \wid,2) {};
\node[v] (b4) at (1* \wid,3) {};
\node[v] (b5) at (1* \wid,4) {};
\node[v] (b6) at (1* \wid,5) {};
\node[v] (b7) at (1* \wid,6) {};

\node[v] (c1) at (2* \wid,0) {};
\node[v] (c2) at (2* \wid,1) {};
\node[v] (c3) at (2* \wid,2) {};
\node[v] (c4) at (2* \wid,3) {};
\node[v] (c5) at (2* \wid,4) {};
\node[v] (c6) at (2* \wid,5) {};
\node[v] (c7) at (2* \wid,6) {};

\draw (1.3* \wid,0) node{$m_1$};
\draw (0.75* \wid,1) node{$m_2$};
\draw (0.75* \wid,2) node{$m_3$};
\draw (1.27* \wid,3) node{$m_4$};
\draw (1.55* \wid,4) node{$m_5$};
\draw (1.4* \wid,5) node{$m_6$};
\draw (0.75* \wid,6) node{$m_7$};

\draw[e, dotted] (1.17* \wid,0) to (b1);
\draw[e, dotted] (0.88* \wid,1) to (b2);
\draw[e, dotted] (0.88* \wid,2) to (b3);
\draw[e, dotted] (1.14* \wid,3) to (b4);
\draw[e, dotted] (1.42* \wid,4) to (b5);
\draw[e, dotted] (1.27* \wid,5) to (b6);
\draw[e, dotted] (0.88* \wid,6) to (b7);

\draw[e] (a4) to[out=0, in=180] (b6);
\draw[e] (a1) to[out=0, in=180] (b5);
\draw[e] (a2) to[out=0, in=180] (b1);

\draw[e] (a3) to[out=0, in=0] (a7);
\draw[e] (a5) to[out=0, in=0] (a6);

\draw[e] (b2) to[out=180, in=180] (b3);
\draw[e] (b2) to[out=180, in=180] (b3);
\draw[e] (b4) to[out=180, in=180] (b7);

\draw[e] (c1) to[out=180, in=180] (c5);
\draw[e] (c4) to[out=180, in=180] (c7);

\draw[e] (b1) to[out=0, in=0] (b4);
\draw[e] (b3) to[out=0, in=0] (b2);

\draw[e] (b5) to[out=0, in=180] (c3);
\draw[e] (b7) to[out=0, in=180] (c6);
\draw[e] (b6) to[out=0, in=180] (c2);

\end{tikzpicture}
\quad
\end{center}

The equivalence classes under the relation $\underset{x,y}{\sim}$ are
$$l_1 \sim m_5 \sim r_3, \ l_2 \sim m_1 \sim m_4 \sim m_7 \sim r_6,$$
$$l_3 \sim l_7, \ l_5 \sim l_6, \ r_1 \sim r_5, \ r_4 \sim r_7,\textrm{ and } m_2 \sim m_3.$$ \end{example}

Ridout and Saint Aubin \cite{RidoutSaintAubin} give a geometric action of Temperley-Lieb diagrams on link states, and extend this to an action of the Temperley-Lieb algebra on the formal $R$-linear combinations of link states. They call this module the \emph{link module}, and it coincides with the \emph{cell module} $W(\lambda)$ defined by Graham and Lehrer \cite{GrahamLehrer}. We will not use either machinery explicitly, but the same ideas will always be in the background - recall especially that for the Brauer algebra our definitions do not coincide with the cellular ones (Remark \ref{rmk:cellVsLS}). In particular, the definitions of those papers explain what kind of multiplication the sesqui-diagrams of the next definition are intended to describe.

\begin{definition} \label{defSesqui} \cite{RidoutSaintAubin} Let $e$ be a diagram, and let $p$ be a link state. The \emph{sesqui-diagram} $(p,e)$ has nodes $m_1, \dots, m_n,$ and $r_1, \dots, r_n$, thought of as arranged in two vertical columns, and edges as follows.

We think of $p$ as a right link state, and embed it by mapping its nodes to the $m_j$. We embed $e$ by mapping its left hand nodes to the $m_j$ and its right hand nodes to the $r_j$.

\end{definition}

\begin{example} In $\Brauer_7$, if
\begin{center}
$p =$
\quad
\begin{tikzpicture}[x=1.5cm,y=-.5cm,baseline=-1.05cm]

\def\wid{2}

\node[v] (b1) at (1* \wid,0) {};
\node[v] (b2) at (1* \wid,1) {};
\node[v] (b3) at (1* \wid,2) {};
\node[v] (b4) at (1* \wid,3) {};
\node[v] (b5) at (1* \wid,4) {};
\node[v] (b6) at (1* \wid,5) {};
\node[v] (b7) at (1* \wid,6) {};

\draw[e] (0.5*\wid, 5) to[out=0, in=180] (b6);
\draw[e] (0.5*\wid, 4) to[out=0, in=180] (b5);
\draw[e] (0.5*\wid, 0) to[out=0, in=180] (b1);

\draw[e] (b2) to[out=180, in=180] (b3);
\draw[e] (b2) to[out=180, in=180] (b3);
\draw[e] (b4) to[out=180, in=180] (b7);

\end{tikzpicture}
\quad
and $e=$
\quad
\begin{tikzpicture}[x=1.5cm,y=-.5cm,baseline=-1.05cm]

\def\wid{2}

\node[v] (b1) at (1* \wid,0) {};
\node[v] (b2) at (1* \wid,1) {};
\node[v] (b3) at (1* \wid,2) {};
\node[v] (b4) at (1* \wid,3) {};
\node[v] (b5) at (1* \wid,4) {};
\node[v] (b6) at (1* \wid,5) {};
\node[v] (b7) at (1* \wid,6) {};

\node[v] (c1) at (2* \wid,0) {};
\node[v] (c2) at (2* \wid,1) {};
\node[v] (c3) at (2* \wid,2) {};
\node[v] (c4) at (2* \wid,3) {};
\node[v] (c5) at (2* \wid,4) {};
\node[v] (c6) at (2* \wid,5) {};
\node[v] (c7) at (2* \wid,6) {};

\draw[e] (c1) to[out=180, in=180] (c5);
\draw[e] (c4) to[out=180, in=180] (c7);

\draw[e] (b1) to[out=0, in=0] (b4);
\draw[e] (b3) to[out=0, in=0] (b2);

\draw[e] (b5) to[out=0, in=180] (c3);
\draw[e] (b7) to[out=0, in=180] (c6);
\draw[e] (b6) to[out=0, in=180] (c2);

\end{tikzpicture}
\end{center}

then the sesqui-diagram formed from $x$ and $y$ is

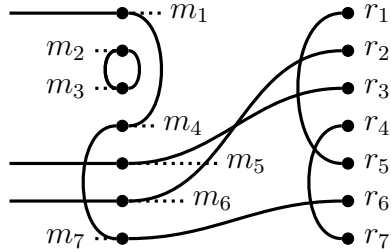
\begin{center}
\quad
\begin{tikzpicture}[x=1.5cm,y=-.5cm,baseline=-1.05cm]

\def\wid{2}

\fill (2*\wid,0) circle (.75mm) node[right=2pt](a) {$r_1$};
\fill (2*\wid,1) circle (.75mm) node[right=2pt](a) {$r_2$};
\fill (2*\wid,2) circle (.75mm) node[right=2pt](a) {$r_3$};
\fill (2*\wid,3) circle (.75mm) node[right=2pt](a) {$r_4$};
\fill (2*\wid,4) circle (.75mm) node[right=2pt](a) {$r_5$};
\fill (2*\wid,5) circle (.75mm) node[right=2pt](a) {$r_6$};
\fill (2*\wid,6) circle (.75mm) node[right=2pt](a) {$r_7$};

\node[v] (b1) at (1* \wid,0) {};
\node[v] (b2) at (1* \wid,1) {};
\node[v] (b3) at (1* \wid,2) {};
\node[v] (b4) at (1* \wid,3) {};
\node[v] (b5) at (1* \wid,4) {};
\node[v] (b6) at (1* \wid,5) {};
\node[v] (b7) at (1* \wid,6) {};

\node[v] (c1) at (2* \wid,0) {};
\node[v] (c2) at (2* \wid,1) {};
\node[v] (c3) at (2* \wid,2) {};
\node[v] (c4) at (2* \wid,3) {};
\node[v] (c5) at (2* \wid,4) {};
\node[v] (c6) at (2* \wid,5) {};
\node[v] (c7) at (2* \wid,6) {};

\draw (1.3* \wid,0) node{$m_1$};
\draw (0.75* \wid,1) node{$m_2$};
\draw (0.75* \wid,2) node{$m_3$};
\draw (1.27* \wid,3) node{$m_4$};
\draw (1.55* \wid,4) node{$m_5$};
\draw (1.4* \wid,5) node{$m_6$};
\draw (0.75* \wid,6) node{$m_7$};

\draw[e, dotted] (1.17* \wid,0) to (b1);
\draw[e, dotted] (0.88* \wid,1) to (b2);
\draw[e, dotted] (0.88* \wid,2) to (b3);
\draw[e, dotted] (1.14* \wid,3) to (b4);
\draw[e, dotted] (1.42* \wid,4) to (b5);
\draw[e, dotted] (1.27* \wid,5) to (b6);
\draw[e, dotted] (0.88* \wid,6) to (b7);

\draw[e] (0.5*\wid, 5) to[out=0, in=180] (b6);
\draw[e] (0.5*\wid, 4) to[out=0, in=180] (b5);
\draw[e] (0.5*\wid, 0) to[out=0, in=180] (b1);

\draw[e] (b2) to[out=180, in=180] (b3);
\draw[e] (b2) to[out=180, in=180] (b3);
\draw[e] (b4) to[out=180, in=180] (b7);

\draw[e] (c1) to[out=180, in=180] (c5);
\draw[e] (c4) to[out=180, in=180] (c7);

\draw[e] (b1) to[out=0, in=0] (b4);
\draw[e] (b3) to[out=0, in=0] (b2);

\draw[e] (b5) to[out=0, in=180] (c3);
\draw[e] (b7) to[out=0, in=180] (c6);
\draw[e] (b6) to[out=0, in=180] (c2);

\end{tikzpicture}
\quad
\end{center}
\end{example}

Unsurprisingly, we think of sesqui-diagrams as being obtained from double diagrams by restricting the left half to its right link state. Precisely, if $x$ and $y$ are diagrams, and $p$ is the right link state of $x$, then the double diagram $(x,y)$ extends the defects from the sesqui-diagram $(p,y)$ to edges, which terminate on the left hand side of the embedded copy of $x$. Relative to $\underset{p,y}{\sim}$, the equivalence relation $\underset{x,y}{\sim}$ has extra equivalence classes (each of cardinality 2, coming from left-to-left connections in $x$) and each left-to-right connection in $x$ enlarges one of the equivalence classes of $\underset{p,y}{\sim}$ by appending a single $l_j$. In particular, we have the following proposition.

\begin{proposition} \label{doubleToSesqui} Let $x$ and $y$ be diagrams, and let $p$ be the right link state of $x$. The restriction of the equivalence relation $\underset{x,y}{\sim}$ to the nodes $\{m_1, \dots, m_n, r_1, \dots, r_n\}$ is precisely $\underset{p,y}{\sim}$. \qed \end{proposition}

The next lemma will be used to prove both of the results of this section, Theorems \ref{generalisedSroka} and \ref{generalisedBrauerSroka}. It give  diagrammatic conditions under which a property stronger than the input to Theorem \ref{subalgebraOfBrauer} holds.

\begin{lemma} \label{LSControl} Let $p \in P_i$ be a link state. Suppose that $e \in \Brauer_n(\delta)$ satisfies the following properties:
\begin{enumerate}
    \item $e$ has right link state $p$, \label{eT1}
    \item $m_j \underset{p,e}{\sim} r_j$ whenever $p$ has a defect at node $j$, and \label{eT2}
    \item for each $j$, there exists some $k$ with $m_j \underset{p,e}{\sim} r_k$. \label{eT3}
\end{enumerate}
Then $y e = y$ for all $y \in J_p$. \end{lemma}

\begin{proof} We must show that the underlying diagram of $y e$ is just $y$, and that the multiplication produces no factors of $\delta$. We will show the equivalent statement that the double diagram $(y,e)$ has all of the connections of $y$, and no loops.

The right link state of $y$ is obtained from $p$ via a sequence of {\mOne}s. The equivalence relation on the set of nodes $\{m_1, \dots, m_n, r_1, \dots, r_n\}$ obtained by restricting $\underset{y,e}{\sim}$ is therefore stronger than $\underset{p,e}{\sim}$ by Proposition \ref{doubleToSesqui}. In other words, nodes connected in the sesqui-diagram $(p,e)$ are connected in the double diagram $(y,e)$. We will use this fact frequently and without comment in the rest of the proof.

Consider some node $r_j$ ($1 \leq j \leq n$) on the right of the double diagram $(y,e)$. We will argue that $r_j$ must be attached to the same place in $(y,e)$ as it is in $y$.

\begin{itemize}
    \item If $j$ is connected to some node $k \neq j$ via a right-to-right connection in $y$, then we must show that $r_j \underset{y,e}{\sim} r_k$. There are two cases to consider.
    \begin{itemize}
        \item If $p$ does not have a defect at height $j$, then $j$ is already connected to $k$ in $e$. By Property (\ref{eT1}), $e$ has right link state $p$, so $r_j \underset{p,e}{\sim} r_k$, so $r_j \underset{y,e}{\sim} r_k$, as required.
        \item If $p$ has a defect at height $j$, then by assumption, $r_j \underset{p,e}{\sim} m_j$, and $j$ is one of the `extra' right-to-right connections that occurs in $y$ but not in $p$. In particular, the node $k$ on the right of $y$ to which $j$ is connected must also be a defect in $p$, so $m_k \underset{p,e}{\sim} r_k$. Thus, $r_j \underset{y,e}{\sim} m_j \underset{y,e}{\sim} m_k \underset{y,e}{\sim} r_k$, as required.
    \end{itemize}
    \item If $j$ is connected to some node $k$ via a left-to-right connection in $y$ (necessarily with $k$ on the left), then we must show that $l_k \underset{y,e}{\sim} r_j$. Now, $j$ is a defect in the right link state of $y$, so since the right link state of $y$ is obtained from $p$ by a sequence of {\mOne}s, $j$ must also be a defect in $p$. By Property (\ref{eT2}), we have $m_j \underset{p,e}{\sim} r_j$, so $m_j \underset{y,e}{\sim} r_j$, so $l_k \underset{y,e}{\sim} m_j \underset{y,e}{\sim} r_j$, where the penultimate equivalence uses that $k$ is connected to $j$ via a left-to-right connection in $y$, as required.
\end{itemize}

Thus, nodes on the right are attached to the same place in $y$ as they are in $(y,e)$. This means that $y$ and $(y,e)$ have the same left-to-right and right-to-right connections. Since $(y,e)$ automatically retains all left-to-left connections from $y$, the two have the same connections, and it remains only to check that there are no loops.

A loop in $(y,e)$ must only pass through nodes from among the middle nodes $m_j$, since these are the only nodes that can have valence greater than one. It therefore suffices to establish that each middle node in $(y,e)$ is connected to some left or right node, and Property (\ref{eT3}) says that in fact each middle node is connected to a right node, as required.

This completes the proof. \end{proof}

\subsection{Proof of Theorem \ref{generalisedBrauerSroka}} \label{subsectionSrokaBrauer}

This subsection is devoted to the proof of Theorem \ref{generalisedBrauerSroka}. The plan is to prove Lemma \ref{easyTrundle}, which will give us the input to Lemma \ref{LSControl}, which in turn feeds Theorem \ref{subalgebraOfBrauer}.

\begin{lemma} \label{easyTrundle} Let $p \in P_i$ be a link state with no missing edges, and at least one defect ($i \geq 1$). There exists a diagram $e = e_p \in \Brauer_n(\delta)$ satisfying Properties (\ref{eT1})-(\ref{eT3}) of Lemma \ref{LSControl}. \end{lemma}

\begin{remark} In the language of \cite{RidoutSaintAubin}, the proof will also show that in the link module we have $p e_p = p$. In interpreting this statement, beware that our generalisation of the link module to the Brauer algebra is not the cellular one. \end{remark}

\begin{proof} We will build the desired diagram $p$ `in situ'. Form a `partial sesqui-diagram', with $p$ as the link state part, and $p$ as the right link state of the diagram part, which is otherwise left blank. Below we illustrate an example in $\Brauer_7$, which we will carry along through the proof:

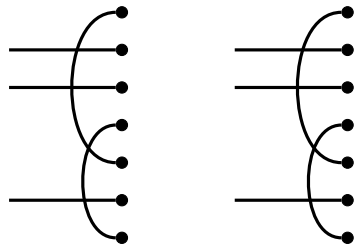
\begin{center}
\quad
\begin{tikzpicture}[x=1.5cm,y=-.5cm,baseline=-1.05cm]

\def\wid{2}

\node[v] (b1) at (1* \wid,0) {};
\node[v] (b2) at (1* \wid,1) {};
\node[v] (b3) at (1* \wid,2) {};
\node[v] (b4) at (1* \wid,3) {};
\node[v] (b5) at (1* \wid,4) {};
\node[v] (b6) at (1* \wid,5) {};
\node[v] (b7) at (1* \wid,6) {};

\node[v] (c1) at (2* \wid,0) {};
\node[v] (c2) at (2* \wid,1) {};
\node[v] (c3) at (2* \wid,2) {};
\node[v] (c4) at (2* \wid,3) {};
\node[v] (c5) at (2* \wid,4) {};
\node[v] (c6) at (2* \wid,5) {};
\node[v] (c7) at (2* \wid,6) {};

\draw[e] (0.5*\wid,5) to[out=0, in=180] (b6);
\draw[e] (0.5*\wid,1) to[out=0, in=180] (b2);
\draw[e] (0.5*\wid,2) to[out=0, in=180] (b3);

\draw[e] (b1) to[out=180, in=180] (b5);
\draw[e] (b4) to[out=180, in=180] (b7);

\draw[e] (c1) to[out=180, in=180] (c5);
\draw[e] (c4) to[out=180, in=180] (c7);
\draw[e] (1.5*\wid,5) to[out=0, in=180] (c6);
\draw[e] (1.5*\wid,1) to[out=0, in=180] (c2);
\draw[e] (1.5*\wid,2) to[out=0, in=180] (c3);

\end{tikzpicture}
\quad
\end{center}

We now complete the right half of the sesqui-diagram to an diagram $e_p$ which will have the desired properties. Note that Property (\ref{eT1}) is already satisfied, because we have fixed the right link state of $e_p$.

First, extend all but one of the defects of $p$ to horizontal edges (this is a laziness that will not be possible in the Temperley-Lieb case). This almost ensures Property (\ref{eT2}): namely, we see that $r_j \underset{p,e}{\sim} m_j$ in the sesqui-diagram at all defects but one, which we call $j_0$ (the existence of this defect is where we use the assumption $i \geq 1$). For our example, choosing the top and bottom defects, we get:

\begin{center}
\quad
\begin{tikzpicture}[x=1.5cm,y=-.5cm,baseline=-1.05cm]

\def\wid{2}

\node[v] (b1) at (1* \wid,0) {};
\node[v] (b2) at (1* \wid,1) {};
\node[v] (b3) at (1* \wid,2) {};
\node[v] (b4) at (1* \wid,3) {};
\node[v] (b5) at (1* \wid,4) {};
\node[v] (b6) at (1* \wid,5) {};
\node[v] (b7) at (1* \wid,6) {};

\node[v] (c1) at (2* \wid,0) {};
\node[v] (c2) at (2* \wid,1) {};
\node[v] (c3) at (2* \wid,2) {};
\node[v] (c4) at (2* \wid,3) {};
\node[v] (c5) at (2* \wid,4) {};
\node[v] (c6) at (2* \wid,5) {};
\node[v] (c7) at (2* \wid,6) {};

\draw[e] (0.5*\wid,5) to[out=0, in=180] (b6);
\draw[e] (0.5*\wid,1) to[out=0, in=180] (b2);
\draw[e] (0.5*\wid,2) to[out=0, in=180] (b3);

\draw[e] (b1) to[out=180, in=180] (b5);
\draw[e] (b4) to[out=180, in=180] (b7);

\draw[e] (c1) to[out=180, in=180] (c5);
\draw[e] (c4) to[out=180, in=180] (c7);
\draw[e] (1.5*\wid,5) to[out=0, in=180] (c6);
\draw[e] (1.5*\wid,1) to[out=0, in=180] (c2);
\draw[e] (1.5*\wid,2) to[out=0, in=180] (c3);

\draw[e, red] (1.5*\wid,5) to[out=180, in=0] (b6);
\draw[e, red] (1.5*\wid,1) to[out=180, in=0] (b2);

\end{tikzpicture}
\quad
\end{center}

Choose a total order on those vertices in the right half of the sesqui-diagram which are part of right-to-right connections in $p$, such that connected vertices are consecutive. Add edges to $e_p$ connecting consecutive vertices which are not already connected in $p$ (`filling in the gaps'). For our example, there is only a single connection to choose:

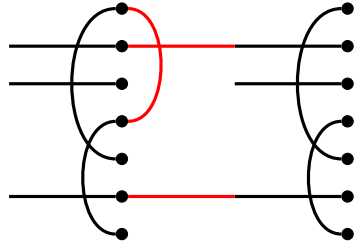
\begin{center}
\quad
\begin{tikzpicture}[x=1.5cm,y=-.5cm,baseline=-1.05cm]

\def\wid{2}

\node[v] (b1) at (1* \wid,0) {};
\node[v] (b2) at (1* \wid,1) {};
\node[v] (b3) at (1* \wid,2) {};
\node[v] (b4) at (1* \wid,3) {};
\node[v] (b5) at (1* \wid,4) {};
\node[v] (b6) at (1* \wid,5) {};
\node[v] (b7) at (1* \wid,6) {};

\node[v] (c1) at (2* \wid,0) {};
\node[v] (c2) at (2* \wid,1) {};
\node[v] (c3) at (2* \wid,2) {};
\node[v] (c4) at (2* \wid,3) {};
\node[v] (c5) at (2* \wid,4) {};
\node[v] (c6) at (2* \wid,5) {};
\node[v] (c7) at (2* \wid,6) {};

\draw[e] (0.5*\wid,5) to[out=0, in=180] (b6);
\draw[e] (0.5*\wid,1) to[out=0, in=180] (b2);
\draw[e] (0.5*\wid,2) to[out=0, in=180] (b3);

\draw[e] (b1) to[out=180, in=180] (b5);
\draw[e] (b4) to[out=180, in=180] (b7);

\draw[e] (c1) to[out=180, in=180] (c5);
\draw[e] (c4) to[out=180, in=180] (c7);
\draw[e] (1.5*\wid,5) to[out=0, in=180] (c6);
\draw[e] (1.5*\wid,1) to[out=0, in=180] (c2);
\draw[e] (1.5*\wid,2) to[out=0, in=180] (c3);

\draw[e, red] (1.5*\wid,5) to[out=180, in=0] (b6);
\draw[e, red] (1.5*\wid,1) to[out=180, in=0] (b2);

\draw[e, red] (b1) to[out=0, in=0] (b4);

\end{tikzpicture}
\quad
\end{center}

This joins the edges in the middle of the diagram into a single connected component (making a Hamiltonian path between all vertices present in right-to-right connections in the left half of the sesqui-diagram), and leaves the two ends of the sequence as-yet unconnected in the right half of the double diagram. In fact, the only nodes in the right half of the double diagram that remain unconnected are these two ends, together with $r_{j_0}$, and $m_{j_0}$.

To complete, connect one end of this sequence to $r_{j_0}$ and connect the other to $m_{j_0}$ (it does not matter which way round):

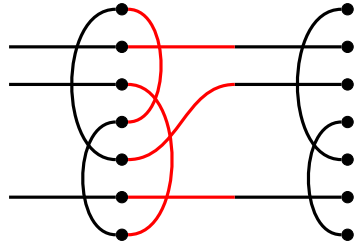
\begin{center}
\quad
\begin{tikzpicture}[x=1.5cm,y=-.5cm,baseline=-1.05cm]

\def\wid{2}

\node[v] (b1) at (1* \wid,0) {};
\node[v] (b2) at (1* \wid,1) {};
\node[v] (b3) at (1* \wid,2) {};
\node[v] (b4) at (1* \wid,3) {};
\node[v] (b5) at (1* \wid,4) {};
\node[v] (b6) at (1* \wid,5) {};
\node[v] (b7) at (1* \wid,6) {};

\node[v] (c1) at (2* \wid,0) {};
\node[v] (c2) at (2* \wid,1) {};
\node[v] (c3) at (2* \wid,2) {};
\node[v] (c4) at (2* \wid,3) {};
\node[v] (c5) at (2* \wid,4) {};
\node[v] (c6) at (2* \wid,5) {};
\node[v] (c7) at (2* \wid,6) {};

\draw[e] (0.5*\wid,5) to[out=0, in=180] (b6);
\draw[e] (0.5*\wid,1) to[out=0, in=180] (b2);
\draw[e] (0.5*\wid,2) to[out=0, in=180] (b3);

\draw[e] (b1) to[out=180, in=180] (b5);
\draw[e] (b4) to[out=180, in=180] (b7);

\draw[e] (c1) to[out=180, in=180] (c5);
\draw[e] (c4) to[out=180, in=180] (c7);
\draw[e] (1.5*\wid,5) to[out=0, in=180] (c6);
\draw[e] (1.5*\wid,1) to[out=0, in=180] (c2);
\draw[e] (1.5*\wid,2) to[out=0, in=180] (c3);

\draw[e, red] (1.5*\wid,5) to[out=180, in=0] (b6);
\draw[e, red] (1.5*\wid,1) to[out=180, in=0] (b2);

\draw[e, red] (b1) to[out=0, in=0] (b4);

\draw[e, red] (b3) to[out=0, in=0] (b7);
\draw[e, red] (b5) to[out=0, in=180] (1.5* \wid, 2);

\end{tikzpicture}
\quad
\end{center}

The result is that in the double diagram, $r_{j_0}$ is connected to $m_{j_0}$ (establishing Property (\ref{eT2}) in the only remaining case, $j=j_0$) \emph{via} the sequence of edges in the middle. To see that each node $m_j$ in the middle is connected to one on the right and establish Property (\ref{eT3}), note that such a node is either at a height $j \neq j_0$ where $p$ has a defect, hence is connected directly to $r_j$, or occurs as part of the `Eulerian cycle' of edges, hence is connected to $r_{j_0}$. This completes the proof. \end{proof}

With Lemma \ref{easyTrundle} in hand, the proof of Theorem \ref{generalisedBrauerSroka} proceeds just like that of Theorem \ref{BrauerRecovery}, except that we must take $\ell = 1$ in Theorem \ref{subalgebraOfBrauer}.

\begin{proof}[Proof of Theorem \ref{generalisedBrauerSroka}] Since $\Brauer_n(\delta)$ is free as an $R$-module on a subset of the rook-Brauer diagrams, we may attempt to apply Theorem \ref{subalgebraOfBrauer} with $\ell = 1$ and $m = n-1$. This will give the correct conclusion, since $I_{\ell-1}=I_0$, and $\faktor{\Brauer_n(\delta)}{\Brauer_n(\delta) \cap I_{n-1}} \cong R \Sigma_n$.

To verify the hypothesis of that theorem, take $i$ in the range $1 \leq i \leq n-1$. Each link state with no missing edges occurs as the right link state of some Brauer diagram, so we must take $p \in P_i$ with no missing edges, and find an idempotent $e_p$ generating $\Brauer_n(\delta) \cap J_p$ as a left ideal.

Let $e_p$ be the diagram given by Lemma \ref{easyTrundle}. By Lemma \ref{LSControl} with $y=e_p$, the diagram $e_p$ is idempotent. Now, the left ideal generated by $e_p$ is in particular closed under $R$-linear combinations, and $\Brauer_n(\delta) \cap J_p$ is free on a basis of diagrams, so it suffices to show that for each diagram $y$ in $\Brauer_n(\delta) \cap J_p$, there exists $x$ in $\Brauer_n(\delta)$ such that $x e_p = y$, but by Lemma \ref{LSControl}, taking $x=y$ satisfies this equation. This completes the proof. \end{proof}

\subsection{Proof of Theorem \ref{generalisedSroka}} \label{subsectionSrokaTL}

We will proceed essentially the same way as in Subsection \ref{subsectionSrokaBrauer}: we still use Theorem \ref{subalgebraOfBrauer}, and we still use Lemma \ref{LSControl}, but now we use Lemma \ref{hardTrundle} in place of Lemma \ref{easyTrundle}. Roughly, we do the same thing, only now we have to be substantially more careful, in order to avoid violating planarity. We will `fill in the edges' in Lemma \ref{hardTrundle} by an argument where we first solve a `local' problem (Lemma \ref{hardSingleTrundle}) and then argue that this gives a solution of the `global' problem (Lemma \ref{spheresOfInfluence}).

Many proofs of Lemma \ref{hardTrundle}, at varying levels of rigour, appear to be possible, and many of them are great fun - the reader may enjoy trying to find their own proof.

We will say that a link state is \emph{planar} if it occurs as the right link state of a Temperley-Lieb diagram, or equivalently if it has no missing edges and no defect lies inside the arc of any connection. A link state is planar if and only if it occurs as the link state of a planar diagram.

Let $p$ be a planar link state with $i \geq 1$ defects, and let $d_i$ be the $i$-th defect. Since no connection may cross a defect, setting $a_j = d_{j+1}$ for $j=1, \dots , i-1$ verifies the following proposition (which typically holds with many other choices of $a_j$).

\begin{proposition} \label{spheresOfInfluence} Let $p$ be a planar link state with $i \geq 1$ defects at nodes $d_1 < d_2 < \dots < d_i$. There exist nodes $1=a_0 < a_1 < \dots < a_i = n+1$ in the set $\{1, \dots, n+1 \}$, such that
\begin{enumerate}
    \item The $j$-th defect $d_j$ lies in the interval $[a_{j-1},a_j)$, and
    \item there are no connections between intervals $[a_{j-1},a_j)$ and $[a_{k-1},a_k)$ for $j \neq k$. \qed
\end{enumerate} \end{proposition}

In words, the intervals $[a_{j-1},a_j)$ give a partition of the set $\{1, \dots, n \}$ into $i$ sets consisting of consecutive numbers, each containing a single defect, and such that there are no connections between these sets.

Given link states $p$ and $q$, we may form the \emph{juxtaposition} of $p$ and $q$. This is a graph with $n$ nodes, arranged vertically, with the connections of $p$ on the left, and the connections of $q$ on the right. This is the graph involved in the inner product of \cite{RidoutSaintAubin} (equivalently, since we are now discussing the Temperley-Lieb algebra, in Graham and Lehrer's \cite{GrahamLehrer} inner product $\phi_\lambda$ on the cell module $W(\lambda)$), though again we will not use this technology explicitly.

\begin{example} In $\TL_7$, if

\begin{center}
$p =$
\quad
\begin{tikzpicture}[x=1.5cm,y=-.5cm,baseline=-1.05cm]

\def\wid{2}

\node[v] (b1) at (1* \wid,0) {};
\node[v] (b2) at (1* \wid,1) {};
\node[v] (b3) at (1* \wid,2) {};
\node[v] (b4) at (1* \wid,3) {};
\node[v] (b5) at (1* \wid,4) {};
\node[v] (b6) at (1* \wid,5) {};
\node[v] (b7) at (1* \wid,6) {};

\draw[e] (0.5*\wid,2) to[out=0, in=180] (b3);

\draw[e] (b1) to[out=180, in=180] (b2);
\draw[e] (b4) to[out=180, in=180] (b7);
\draw[e] (b5) to[out=180, in=180] (b6);

\end{tikzpicture}
\quad
\textrm{ and } $q=$
\quad
\begin{tikzpicture}[x=1.5cm,y=-.5cm,baseline=-1.05cm]

\def\wid{2}

\node[v] (b1) at (1* \wid,0) {};
\node[v] (b2) at (1* \wid,1) {};
\node[v] (b3) at (1* \wid,2) {};
\node[v] (b4) at (1* \wid,3) {};
\node[v] (b5) at (1* \wid,4) {};
\node[v] (b6) at (1* \wid,5) {};
\node[v] (b7) at (1* \wid,6) {};

\draw[e] (0.5*\wid,6) to[out=0, in=180] (b7);

\draw[e] (b1) to[out=180, in=180] (b4);
\draw[e] (b2) to[out=180, in=180] (b3);
\draw[e] (b5) to[out=180, in=180] (b6);

\end{tikzpicture}
\quad
\end{center}

then the juxtaposition of $p$ and $q$ is the graph
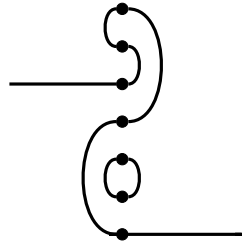
\begin{center}
\quad
\begin{tikzpicture}[x=1.5cm,y=-.5cm,baseline=-1.05cm]

\def\wid{2}

\node[v] (b1) at (1* \wid,0) {};
\node[v] (b2) at (1* \wid,1) {};
\node[v] (b3) at (1* \wid,2) {};
\node[v] (b4) at (1* \wid,3) {};
\node[v] (b5) at (1* \wid,4) {};
\node[v] (b6) at (1* \wid,5) {};
\node[v] (b7) at (1* \wid,6) {};

\draw[e] (0.5*\wid,2) to[out=0, in=180] (b3);

\draw[e] (b1) to[out=180, in=180] (b2);
\draw[e] (b4) to[out=180, in=180] (b7);
\draw[e] (b5) to[out=180, in=180] (b6);

\draw[e] (1.5*\wid,6) to[out=0, in=180] (b7);

\draw[e] (b1) to[out=0, in=0] (b4);
\draw[e] (b2) to[out=0, in=0] (b3);
\draw[e] (b5) to[out=0, in=0] (b6);

\end{tikzpicture}
\quad
\end{center} \end{example}

A statement equivalent to the next lemma was proven in \cite{FGG}.

\begin{lemma} \label{hardSingleTrundle} Let $p$ be a planar link state with no defects and no missing edges. There exists a planar link state $q$ with two defects and no missing edges, such that the juxtaposition of $p$ and $q$ consists of a single connected component. \end{lemma}

\begin{proof} Let $p$ be a planar link state with no defects and no missing edges. In \cite{FGG} it is shown (in the argument beginning at Fig. 21 and justifying Equation 6.30) that there exists a planar link state $q'$ with no defects and no missing edges, such that the juxtaposition of $p$ and $q'$ consists of a single loop: in their language `every upper arch configuration may be extended to a one-component meander'.

This link state $q'$ must have some connection that is not inside the arc of any other (the one connected to the top node will do, for example). We may cut this connection to form a pair of defects (performing the inverse of a \mOne), and since it is not inside the arc of any other connections, this does not violate planarity. The resulting link state $q$ has the required properties. \end{proof}

\begin{lemma} \label{hardTrundle} Let $p \in P_i$ be a planar link state with no missing edges, and at least one defect ($i \geq 1$). There exists a diagram $e_p \in \TL_n(\delta)$ satisfying Properties (\ref{eT1})-(\ref{eT3}) of Lemma \ref{LSControl}. \end{lemma}

\begin{proof} For given $p$, we must construct $e_p$ such that
\begin{enumerate}
    \item $e_p$ has right link state $p$,
    \item $r_j \underset{p,e_p}{\sim} m_j$ whenever $p$ has a defect at node $j$, and
    \item for each $j$, there exists some $k$ with $m_j \underset{p,e_p}{\sim} r_k$.
\end{enumerate}

As in the proof of Lemma \ref{easyTrundle}, form the partial sesqui-diagram with $p$ on the left, and an otherwise-blank diagram having $p$ as its right link state (to be completed to $e_p$) on the right. We will argue similarly to Lemma \ref{easyTrundle}, now also ensuring that we do not add any intersecting connections. Below we illustrate an example from $\TL_{15}$, which we will carry along through the proof

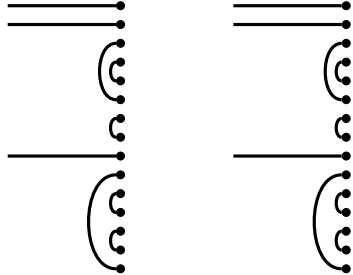
\begin{center}
\quad
\begin{tikzpicture}[x=1.5cm,y=-.5cm,baseline=-1.05cm]

\def\wid{2}
\def\hei{0.5}
\def\nodesize{3}

\node[v, minimum size=\nodesize] (b1) at (1* \wid,0*\hei) {};
\node[v, minimum size=\nodesize] (b2) at (1* \wid,1*\hei) {};
\node[v, minimum size=\nodesize] (b3) at (1* \wid,2*\hei) {};
\node[v, minimum size=\nodesize] (b4) at (1* \wid,3*\hei) {};
\node[v, minimum size=\nodesize] (b5) at (1* \wid,4*\hei) {};
\node[v, minimum size=\nodesize] (b6) at (1* \wid,5*\hei) {};
\node[v, minimum size=\nodesize] (b7) at (1* \wid,6*\hei) {};
\node[v, minimum size=\nodesize] (b8) at (1* \wid,7*\hei) {};
\node[v, minimum size=\nodesize] (b9) at (1* \wid,8*\hei) {};
\node[v, minimum size=\nodesize] (b10) at (1* \wid,9*\hei) {};
\node[v, minimum size=\nodesize] (b11) at (1* \wid,10*\hei) {};
\node[v, minimum size=\nodesize] (b12) at (1* \wid,11*\hei) {};
\node[v, minimum size=\nodesize] (b13) at (1* \wid,12*\hei) {};
\node[v, minimum size=\nodesize] (b14) at (1* \wid,13*\hei) {};
\node[v, minimum size=\nodesize] (b15) at (1* \wid,14*\hei) {};

\node[v, minimum size=\nodesize] (c1) at (2* \wid,0*\hei) {};
\node[v, minimum size=\nodesize] (c2) at (2* \wid,1*\hei) {};
\node[v, minimum size=\nodesize] (c3) at (2* \wid,2*\hei) {};
\node[v, minimum size=\nodesize] (c4) at (2* \wid,3*\hei) {};
\node[v, minimum size=\nodesize] (c5) at (2* \wid,4*\hei) {};
\node[v, minimum size=\nodesize] (c6) at (2* \wid,5*\hei) {};
\node[v, minimum size=\nodesize] (c7) at (2* \wid,6*\hei) {};
\node[v, minimum size=\nodesize] (c8) at (2* \wid,7*\hei) {};
\node[v, minimum size=\nodesize] (c9) at (2* \wid,8*\hei) {};
\node[v, minimum size=\nodesize] (c10) at (2* \wid,9*\hei) {};
\node[v, minimum size=\nodesize] (c11) at (2* \wid,10*\hei) {};
\node[v, minimum size=\nodesize] (c12) at (2* \wid,11*\hei) {};
\node[v, minimum size=\nodesize] (c13) at (2* \wid,12*\hei) {};
\node[v, minimum size=\nodesize] (c14) at (2* \wid,13*\hei) {};
\node[v, minimum size=\nodesize] (c15) at (2* \wid,14*\hei) {};

\draw[e] (0.5*\wid,0*\hei) to[out=0, in=180] (b1);
\draw[e] (0.5*\wid,1*\hei) to[out=0, in=180] (b2);
\draw[e] (0.5*\wid,8*\hei) to[out=0, in=180] (b9);

\draw[e] (b3) to[out=180, in=180] (b6);
\draw[e] (b4) to[out=180, in=180] (b5);
\draw[e] (b7) to[out=180, in=180] (b8);

\draw[e] (b10) to[out=180, in=180] (b15);
\draw[e] (b11) to[out=180, in=180] (b12);
\draw[e] (b13) to[out=180, in=180] (b14);

\draw[e] (1.5*\wid,0*\hei) to[out=0, in=180] (c1);
\draw[e] (1.5*\wid,1*\hei) to[out=0, in=180] (c2);
\draw[e] (1.5*\wid,8*\hei) to[out=0, in=180] (c9);

\draw[e] (c3) to[out=180, in=180] (c6);
\draw[e] (c4) to[out=180, in=180] (c5);
\draw[e] (c7) to[out=180, in=180] (c8);

\draw[e] (c10) to[out=180, in=180] (c15);
\draw[e] (c11) to[out=180, in=180] (c12);
\draw[e] (c13) to[out=180, in=180] (c14);




\end{tikzpicture}
\quad
\end{center}

The goal is to complete the edges of $e_p$ so that between them the defects make a `Eulerian multi-cycle' through all of the edges on the left. As before, this will automatically satisfy Property (\ref{eT1}).

We first establish `spheres of influence' for each of the defects, so that each can be treated separately. Let $d_1 < \dots < d_i$ be the positions of the defects. By Proposition \ref{spheresOfInfluence}, since $i \geq 1$, there exist integers $0=a_0< \dots < a_i=n+1$ such that \begin{enumerate}
    \item The $j$-th defect $d_j$ lies in the interval $[a_{j-1},a_j)$, and
    \item there are no connections between intervals $[a_{j-1},a_j)$ and $[a_{k-1},a_k)$ for $j \neq k$.
\end{enumerate}
We will call the interval $[a_{j-1},a_j)$ the \emph{garden belonging to the $j$-th defect}. For our example, $i=3$, and we might get $a_0 = 1$, $a_1=2$, $a_2=7$, $a_3=16$. Adding dotted lines to indicate the boundaries between gardens, this produces:
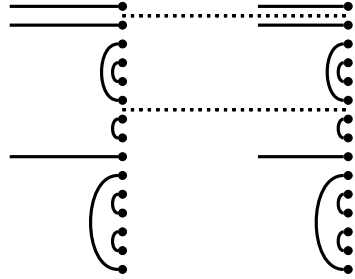
\begin{center}
\quad
\begin{tikzpicture}[x=1.5cm,y=-.5cm,baseline=-1.05cm]

\def\wid{2}
\def\hei{0.5}
\def\nodesize{3}

\draw[e, dotted] (\wid,0.5*\hei) rectangle (2*\wid,0.5*\hei);
\draw[e, dotted] (\wid,5.5*\hei) rectangle (2*\wid,5.5*\hei);


\node[v, minimum size=\nodesize] (b1) at (1* \wid,0*\hei) {};
\node[v, minimum size=\nodesize] (b2) at (1* \wid,1*\hei) {};
\node[v, minimum size=\nodesize] (b3) at (1* \wid,2*\hei) {};
\node[v, minimum size=\nodesize] (b4) at (1* \wid,3*\hei) {};
\node[v, minimum size=\nodesize] (b5) at (1* \wid,4*\hei) {};
\node[v, minimum size=\nodesize] (b6) at (1* \wid,5*\hei) {};
\node[v, minimum size=\nodesize] (b7) at (1* \wid,6*\hei) {};
\node[v, minimum size=\nodesize] (b8) at (1* \wid,7*\hei) {};
\node[v, minimum size=\nodesize] (b9) at (1* \wid,8*\hei) {};
\node[v, minimum size=\nodesize] (b10) at (1* \wid,9*\hei) {};
\node[v, minimum size=\nodesize] (b11) at (1* \wid,10*\hei) {};
\node[v, minimum size=\nodesize] (b12) at (1* \wid,11*\hei) {};
\node[v, minimum size=\nodesize] (b13) at (1* \wid,12*\hei) {};
\node[v, minimum size=\nodesize] (b14) at (1* \wid,13*\hei) {};
\node[v, minimum size=\nodesize] (b15) at (1* \wid,14*\hei) {};

\node[v, minimum size=\nodesize] (c1) at (2* \wid,0*\hei) {};
\node[v, minimum size=\nodesize] (c2) at (2* \wid,1*\hei) {};
\node[v, minimum size=\nodesize] (c3) at (2* \wid,2*\hei) {};
\node[v, minimum size=\nodesize] (c4) at (2* \wid,3*\hei) {};
\node[v, minimum size=\nodesize] (c5) at (2* \wid,4*\hei) {};
\node[v, minimum size=\nodesize] (c6) at (2* \wid,5*\hei) {};
\node[v, minimum size=\nodesize] (c7) at (2* \wid,6*\hei) {};
\node[v, minimum size=\nodesize] (c8) at (2* \wid,7*\hei) {};
\node[v, minimum size=\nodesize] (c9) at (2* \wid,8*\hei) {};
\node[v, minimum size=\nodesize] (c10) at (2* \wid,9*\hei) {};
\node[v, minimum size=\nodesize] (c11) at (2* \wid,10*\hei) {};
\node[v, minimum size=\nodesize] (c12) at (2* \wid,11*\hei) {};
\node[v, minimum size=\nodesize] (c13) at (2* \wid,12*\hei) {};
\node[v, minimum size=\nodesize] (c14) at (2* \wid,13*\hei) {};
\node[v, minimum size=\nodesize] (c15) at (2* \wid,14*\hei) {};

\draw[e] (0.5*\wid,0*\hei) to[out=0, in=180] (b1);
\draw[e] (0.5*\wid,1*\hei) to[out=0, in=180] (b2);
\draw[e] (0.5*\wid,8*\hei) to[out=0, in=180] (b9);

\draw[e] (b3) to[out=180, in=180] (b6);
\draw[e] (b4) to[out=180, in=180] (b5);
\draw[e] (b7) to[out=180, in=180] (b8);

\draw[e] (b10) to[out=180, in=180] (b15);
\draw[e] (b11) to[out=180, in=180] (b12);
\draw[e] (b13) to[out=180, in=180] (b14);

\draw[e] (1.6*\wid,0*\hei) to[out=0, in=180] (c1);
\draw[e] (1.6*\wid,1*\hei) to[out=0, in=180] (c2);
\draw[e] (1.6*\wid,8*\hei) to[out=0, in=180] (c9);

\draw[e] (c3) to[out=180, in=180] (c6);
\draw[e] (c4) to[out=180, in=180] (c5);
\draw[e] (c7) to[out=180, in=180] (c8);

\draw[e] (c10) to[out=180, in=180] (c15);
\draw[e] (c11) to[out=180, in=180] (c12);
\draw[e] (c13) to[out=180, in=180] (c14);




\end{tikzpicture}
\quad
\end{center}

Each defect $d_j$ may be strictly internal to its garden, or may lie at one end of it. The nodes in $[a_{j-1},a_j)$ that lie above the defect in the diagram (hence have numbers smaller than $d_j$) will be called $d_j$'s \emph{front garden}, and those that lie below $d_j$ will be called $d_j$'s \emph{back garden}. For any given $j$, it is possible that one or both of the gardens are empty.

Fix $j$. The edges in each of $d_j$'s \emph{front garden} (if non-empty) and \emph{back garden} (if non-empty) are a link state on a subset of the nodes, by construction having no defects. Apply Lemma \ref{hardSingleTrundle} to each of these `local' link states, and add the resulting link state to the right hand side of $e_p$. This adds edges to each garden so that the existing edges are connected into a single component (a Hamiltonian path) and so that two nodes are left, neither lying inside the arc of any of the new edges. For the front garden, call these nodes $f_0 < f_1$, and for the back garden call them $b_0 < b_1$. For the example (performing this step for each $j$) this might look as follows:

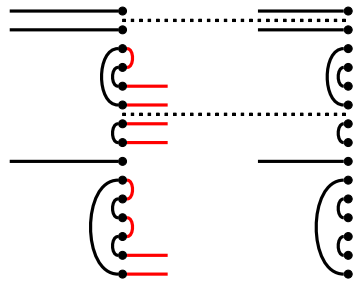
\begin{center}
\quad
\begin{tikzpicture}[x=1.5cm,y=-.5cm,baseline=-1.05cm]

\def\wid{2}
\def\hei{0.5}
\def\nodesize{3}

\draw[e, dotted] (\wid,0.5*\hei) rectangle (2*\wid,0.5*\hei);
\draw[e, dotted] (\wid,5.5*\hei) rectangle (2*\wid,5.5*\hei);


\node[v, minimum size=\nodesize] (b1) at (1* \wid,0*\hei) {};
\node[v, minimum size=\nodesize] (b2) at (1* \wid,1*\hei) {};
\node[v, minimum size=\nodesize] (b3) at (1* \wid,2*\hei) {};
\node[v, minimum size=\nodesize] (b4) at (1* \wid,3*\hei) {};
\node[v, minimum size=\nodesize] (b5) at (1* \wid,4*\hei) {};
\node[v, minimum size=\nodesize] (b6) at (1* \wid,5*\hei) {};
\node[v, minimum size=\nodesize] (b7) at (1* \wid,6*\hei) {};
\node[v, minimum size=\nodesize] (b8) at (1* \wid,7*\hei) {};
\node[v, minimum size=\nodesize] (b9) at (1* \wid,8*\hei) {};
\node[v, minimum size=\nodesize] (b10) at (1* \wid,9*\hei) {};
\node[v, minimum size=\nodesize] (b11) at (1* \wid,10*\hei) {};
\node[v, minimum size=\nodesize] (b12) at (1* \wid,11*\hei) {};
\node[v, minimum size=\nodesize] (b13) at (1* \wid,12*\hei) {};
\node[v, minimum size=\nodesize] (b14) at (1* \wid,13*\hei) {};
\node[v, minimum size=\nodesize] (b15) at (1* \wid,14*\hei) {};

\node[v, minimum size=\nodesize] (c1) at (2* \wid,0*\hei) {};
\node[v, minimum size=\nodesize] (c2) at (2* \wid,1*\hei) {};
\node[v, minimum size=\nodesize] (c3) at (2* \wid,2*\hei) {};
\node[v, minimum size=\nodesize] (c4) at (2* \wid,3*\hei) {};
\node[v, minimum size=\nodesize] (c5) at (2* \wid,4*\hei) {};
\node[v, minimum size=\nodesize] (c6) at (2* \wid,5*\hei) {};
\node[v, minimum size=\nodesize] (c7) at (2* \wid,6*\hei) {};
\node[v, minimum size=\nodesize] (c8) at (2* \wid,7*\hei) {};
\node[v, minimum size=\nodesize] (c9) at (2* \wid,8*\hei) {};
\node[v, minimum size=\nodesize] (c10) at (2* \wid,9*\hei) {};
\node[v, minimum size=\nodesize] (c11) at (2* \wid,10*\hei) {};
\node[v, minimum size=\nodesize] (c12) at (2* \wid,11*\hei) {};
\node[v, minimum size=\nodesize] (c13) at (2* \wid,12*\hei) {};
\node[v, minimum size=\nodesize] (c14) at (2* \wid,13*\hei) {};
\node[v, minimum size=\nodesize] (c15) at (2* \wid,14*\hei) {};

\draw[e] (0.5*\wid,0*\hei) to[out=0, in=180] (b1);
\draw[e] (0.5*\wid,1*\hei) to[out=0, in=180] (b2);
\draw[e] (0.5*\wid,8*\hei) to[out=0, in=180] (b9);

\draw[e] (b3) to[out=180, in=180] (b6);
\draw[e] (b4) to[out=180, in=180] (b5);
\draw[e] (b7) to[out=180, in=180] (b8);

\draw[e] (b10) to[out=180, in=180] (b15);
\draw[e] (b11) to[out=180, in=180] (b12);
\draw[e] (b13) to[out=180, in=180] (b14);

\draw[e] (1.6*\wid,0*\hei) to[out=0, in=180] (c1);
\draw[e] (1.6*\wid,1*\hei) to[out=0, in=180] (c2);
\draw[e] (1.6*\wid,8*\hei) to[out=0, in=180] (c9);

\draw[e] (c3) to[out=180, in=180] (c6);
\draw[e] (c4) to[out=180, in=180] (c5);
\draw[e] (c7) to[out=180, in=180] (c8);

\draw[e] (c10) to[out=180, in=180] (c15);
\draw[e] (c11) to[out=180, in=180] (c12);
\draw[e] (c13) to[out=180, in=180] (c14);

\draw[e, red] (b3) to[out=0, in=0] (b4);
\draw[e, red] (1.2* \wid,4*\hei) to[out=180, in=0] (b5);
\draw[e, red] (1.2* \wid,5*\hei) to[out=180, in=0] (b6);

\draw[e, red] (1.2* \wid,6*\hei) to[out=180, in=0] (b7);
\draw[e, red] (1.2* \wid,7*\hei) to[out=180, in=0] (b8);

\draw[e, red] (1.2* \wid,13*\hei) to[out=180, in=0] (b14);
\draw[e, red] (1.2* \wid,14*\hei) to[out=180, in=0] (b15);
\draw[e, red] (b10) to[out=0, in=0] (b11);
\draw[e, red] (b12) to[out=0, in=0] (b13);

\end{tikzpicture}
\quad
\end{center}

We will now connect the left and right nodes at height $d_j$ via a `Eulerian cycle through both gardens'. Because one or both gardens may be empty, a case division is necessary.

\begin{itemize}
    \item If both gardens are empty, then add a single horizontal connection at height $d_j$.
    \item If one garden is empty and the other is non-empty, then we divide further:\begin{itemize}
    \item If only the front garden is non-empty, then connect the right-hand instance of $d_j$ to $f_0$, and connect $f_1$ to the left-hand instance of $d_j$. Because $f_0 < f_1$, these two connections do not intersect.
    \item If only the back garden is non-empty, then connect the right-hand instance of $d_j$ to $b_1$, and connect $b_0$ to the left-hand instance of $d_j$. Because $b_0 < b_1$, these two connections do not intersect.
\end{itemize}
    \item If both gardens are non-empty, then connect the right-hand instance of $d_j$ to $b_1$, connect $b_0$ to $f_0$, and connect $f_1$ to the left-hand instance of $d_j$. One the left, we have $f_0 < f_1 < d_j < b_0 < b_1$. Because $b_1$ is the largest, the horizontal connection from $j$ to $b_1$ does not intersect either of the two left-to left connections. Because $f_0 < f_1 < d_j < b_0$, the two left-to-left connections (are nested and) do not intersect.
\end{itemize}

In the example, applying these steps to each garden gives:

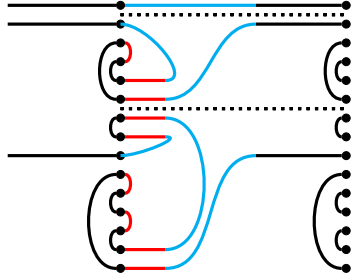
\begin{center}
\quad
\begin{tikzpicture}[x=1.5cm,y=-.5cm,baseline=-1.05cm]

\def\wid{2}
\def\hei{0.5}
\def\nodesize{3}

\draw[e, dotted] (\wid,0.5*\hei) rectangle (2*\wid,0.5*\hei);
\draw[e, dotted] (\wid,5.5*\hei) rectangle (2*\wid,5.5*\hei);


\node[v, minimum size=\nodesize] (b1) at (1* \wid,0*\hei) {};
\node[v, minimum size=\nodesize] (b2) at (1* \wid,1*\hei) {};
\node[v, minimum size=\nodesize] (b3) at (1* \wid,2*\hei) {};
\node[v, minimum size=\nodesize] (b4) at (1* \wid,3*\hei) {};
\node[v, minimum size=\nodesize] (b5) at (1* \wid,4*\hei) {};
\node[v, minimum size=\nodesize] (b6) at (1* \wid,5*\hei) {};
\node[v, minimum size=\nodesize] (b7) at (1* \wid,6*\hei) {};
\node[v, minimum size=\nodesize] (b8) at (1* \wid,7*\hei) {};
\node[v, minimum size=\nodesize] (b9) at (1* \wid,8*\hei) {};
\node[v, minimum size=\nodesize] (b10) at (1* \wid,9*\hei) {};
\node[v, minimum size=\nodesize] (b11) at (1* \wid,10*\hei) {};
\node[v, minimum size=\nodesize] (b12) at (1* \wid,11*\hei) {};
\node[v, minimum size=\nodesize] (b13) at (1* \wid,12*\hei) {};
\node[v, minimum size=\nodesize] (b14) at (1* \wid,13*\hei) {};
\node[v, minimum size=\nodesize] (b15) at (1* \wid,14*\hei) {};

\node[v, minimum size=\nodesize] (c1) at (2* \wid,0*\hei) {};
\node[v, minimum size=\nodesize] (c2) at (2* \wid,1*\hei) {};
\node[v, minimum size=\nodesize] (c3) at (2* \wid,2*\hei) {};
\node[v, minimum size=\nodesize] (c4) at (2* \wid,3*\hei) {};
\node[v, minimum size=\nodesize] (c5) at (2* \wid,4*\hei) {};
\node[v, minimum size=\nodesize] (c6) at (2* \wid,5*\hei) {};
\node[v, minimum size=\nodesize] (c7) at (2* \wid,6*\hei) {};
\node[v, minimum size=\nodesize] (c8) at (2* \wid,7*\hei) {};
\node[v, minimum size=\nodesize] (c9) at (2* \wid,8*\hei) {};
\node[v, minimum size=\nodesize] (c10) at (2* \wid,9*\hei) {};
\node[v, minimum size=\nodesize] (c11) at (2* \wid,10*\hei) {};
\node[v, minimum size=\nodesize] (c12) at (2* \wid,11*\hei) {};
\node[v, minimum size=\nodesize] (c13) at (2* \wid,12*\hei) {};
\node[v, minimum size=\nodesize] (c14) at (2* \wid,13*\hei) {};
\node[v, minimum size=\nodesize] (c15) at (2* \wid,14*\hei) {};

\draw[e] (0.5*\wid,0*\hei) to[out=0, in=180] (b1);
\draw[e] (0.5*\wid,1*\hei) to[out=0, in=180] (b2);
\draw[e] (0.5*\wid,8*\hei) to[out=0, in=180] (b9);

\draw[e] (b3) to[out=180, in=180] (b6);
\draw[e] (b4) to[out=180, in=180] (b5);
\draw[e] (b7) to[out=180, in=180] (b8);

\draw[e] (b10) to[out=180, in=180] (b15);
\draw[e] (b11) to[out=180, in=180] (b12);
\draw[e] (b13) to[out=180, in=180] (b14);

\draw[e] (1.6*\wid,0*\hei) to[out=0, in=180] (c1);
\draw[e] (1.6*\wid,1*\hei) to[out=0, in=180] (c2);
\draw[e] (1.6*\wid,8*\hei) to[out=0, in=180] (c9);

\draw[e] (c3) to[out=180, in=180] (c6);
\draw[e] (c4) to[out=180, in=180] (c5);
\draw[e] (c7) to[out=180, in=180] (c8);

\draw[e] (c10) to[out=180, in=180] (c15);
\draw[e] (c11) to[out=180, in=180] (c12);
\draw[e] (c13) to[out=180, in=180] (c14);

\draw[e, red] (b3) to[out=0, in=0] (b4);
\draw[e, red] (1.2* \wid,4*\hei) to[out=180, in=0] (b5);
\draw[e, red] (1.2* \wid,5*\hei) to[out=180, in=0] (b6);

\draw[e, red] (1.2* \wid,6*\hei) to[out=180, in=0] (b7);
\draw[e, red] (1.2* \wid,7*\hei) to[out=180, in=0] (b8);

\draw[e, red] (1.2* \wid,13*\hei) to[out=180, in=0] (b14);
\draw[e, red] (1.2* \wid,14*\hei) to[out=180, in=0] (b15);
\draw[e, red] (b10) to[out=0, in=0] (b11);
\draw[e, red] (b12) to[out=0, in=0] (b13);

\draw[e, cyan] (1.6* \wid,0*\hei) to[out=180, in=0] (b1);

\draw[e, cyan] (1.6* \wid,1*\hei) to[out=180, in=0] (1.2* \wid,5*\hei);
\draw[e, cyan] (1.2* \wid,4*\hei) to[out=0, in=0] (1* \wid,1*\hei);

\draw[e, cyan] (1.6* \wid,8*\hei) to[out=180, in=0] (1.2* \wid,14*\hei);
\draw[e, cyan] (1.2* \wid,13*\hei) to[out=0, in=0] (1.2* \wid,6*\hei);
\draw[e, cyan] (1.2* \wid,7*\hei) to[out=0, in=0] (1* \wid,8*\hei);

\end{tikzpicture}
\quad
\end{center}

We have now finished the construction of $e_p$.

All nodes in the right hand half of the sesqui-diagram have now been connected to something, so $e_p$ is a diagram with no missing edges. By Lemma \ref{hardSingleTrundle}, the above procedure connects the left- and right-hand instances of $d_j$ for each $j$, establishing Property (\ref{eT2}), and each edge from $p$ on the left of the sesqui-diagram is strictly internal to the garden belonging to some $d_j$, hence, again by Lemma \ref{hardSingleTrundle}, lies on the path between the two instances of $d_j$, establishing Property (\ref{eT1}).

It remains to establish planarity. We must argue that no pair of connections in $e_p$ intersects. We will break into cases, based on the fact that each connection is either left-to-left, right-to-right, or left-to-right.
\begin{itemize}
    \item No right-to-right connection in $e_p$ may intersect another (right-to-right or left-to-right) connection, because we began with the planar link state $p$.
    \item No pair of left-to-right connections may intersect, because each defect $d_j$ is connected to a node in its own garden(s), and if $j<k$ then the garden(s) of $d_j$ lie strictly above the garden(s) of $d_k$.
    \item If a left-to-right connection (from, say, defect $d_j$ on the right) intersects a left-to-left connection, then the left-to-left connection must lie in $d_j$'s own garden, but $d_j$ is connected to either $b_0$, $f_1$, or to the left-hand instance of $d_j$, and none of these points lies inside the arc of any of the left-to-left connections produced by Lemma \ref{hardSingleTrundle}. The remaining left-to-left connections are those among the two instances of $d_j$, the $f_i$ and the $b_i$, and we established during the construction that these do not intersect the horizontal connection.
    \item If two left-to-left connections intersect, then they must belong to the same garden. The connections strictly internal to the front and back gardens are those produced by Lemma \ref{hardSingleTrundle}, which do not intersect by that lemma. We established during the construction that the left-to-left connections among the two instances of $d_j$, the $f_i$ and the $b_i$ do not intersect one another, and these do not intersect any of the other left-to-left connections by Lemma \ref{hardSingleTrundle}.
\end{itemize}

This establishes that $e_p$ is planar, and completes the proof. \end{proof}

We now prove Theorem \ref{generalisedSroka}.

\begin{proof}[Proof of Theorem \ref{generalisedSroka}] Since $\TL_n(\delta)$ is free as an $R$-module on a subset of the rook-Brauer diagrams, we may attempt to apply Theorem \ref{subalgebraOfBrauer} with $\ell = 1$ and $m = n-1$. This will give the correct conclusion, since $I_{\ell-1}=I_0$, and $\faktor{\TL_n(\delta)}{\TL_n(\delta) \cap I_{n-1}} \cong R$.

To verify the hypothesis of that theorem, take $i$ in the range $1 \leq i \leq n-1$. The algebra $\TL_n(\delta)$ contains diagrams possessing each right link state that is planar and has no missing edges, so we must take such a $p \in P_i$, and find an idempotent $e_p$ generating $\TL_n(\delta) \cap J_p$ as a left ideal.

Let $e_p$ be the diagram given by Lemma \ref{hardTrundle}. By Lemma \ref{LSControl} with $y=e_p$, the diagram $e_p$ is idempotent. Now, the left ideal generated by $e_p$ is in particular closed under $R$-linear combinations, and $\TL_n(\delta) \cap J_p$ is free on a basis of diagrams, so it suffices to show that for each diagram $y$ in $\TL_n(\delta) \cap J_p$, there exists $x$ in $\TL_n(\delta)$ such that $x e_p = y$, but by Lemma \ref{LSControl}, taking $x=y$ satisfies this equation. This completes the proof. \end{proof}

\printbibliography

@article{Hepworth,
author = {Hepworth, Richard},
title = {Homological stability for Iwahori–Hecke algebras},
journal = {Journal of Topology},
volume = {15},
number = {4},
pages = {2174-2215},
year = {2022}
}

@article {BH,
    AUTHOR = {Boyd, Rachael and Hepworth, Richard},
     TITLE = {The homology of the {T}emperley-{L}ieb algebras},
   JOURNAL = {Geom. Topol.},
  FJOURNAL = {Geometry \& Topology},
    VOLUME = {28},
      YEAR = {2024},
    NUMBER = {3},
     PAGES = {1437--1499}
}

@article {BHComb,
    AUTHOR = {Boyd, Rachael and Hepworth, Richard},
     TITLE = {Combinatorics of injective words for {T}emperley-{L}ieb
              algebras},
   JOURNAL = {J. Combin. Theory Ser. A},
  FJOURNAL = {Journal of Combinatorial Theory. Series A},
    VOLUME = {181},
      YEAR = {2021},
     PAGES = {Paper No. 105446, 27}
}

@article {BHP,
    AUTHOR = {Boyd, Rachael and Hepworth, Richard and Patzt, Peter},
     TITLE = {The homology of the {B}rauer algebras},
   JOURNAL = {Selecta Math. (N.S.)},
  FJOURNAL = {Selecta Mathematica. New Series},
    VOLUME = {27},
      YEAR = {2021},
    NUMBER = {5},
     PAGES = {Paper No. 85, 31}
}

@article {Sroka,
    AUTHOR = {Sroka, Robin J.},
     TITLE = {The homology of a {T}emperley-{L}ieb algebra on an odd number of strands},
   JOURNAL = {Algebr. Geom. Topol.},
  FJOURNAL = {Algebraic \& Geometric Topology},
    VOLUME = {24},
      YEAR = {2024},
    NUMBER = {6},
     PAGES = {3527--3542}
}

@article {Solomon,
    AUTHOR = {Solomon, Louis},
     TITLE = {Representations of the rook monoid},
   JOURNAL = {J. Algebra},
  FJOURNAL = {Journal of Algebra},
    VOLUME = {256},
      YEAR = {2002},
    NUMBER = {2},
     PAGES = {309--342}
}

@article {HalversondelMas,
    AUTHOR = {Halverson, Tom and delMas, Elise},
     TITLE = {Representations of the {R}ook-{B}rauer algebra},
   JOURNAL = {Comm. Algebra},
  FJOURNAL = {Communications in Algebra},
    VOLUME = {42},
      YEAR = {2014},
    NUMBER = {1},
     PAGES = {423--443}
}

@article {RidoutSaintAubin,
    AUTHOR = {Ridout, David and Saint-Aubin, Yvan},
     TITLE = {Standard modules, induction and the structure of the
              {T}emperley-{L}ieb algebra},
   JOURNAL = {Adv. Theor. Math. Phys.},
  FJOURNAL = {Advances in Theoretical and Mathematical Physics},
    VOLUME = {18},
      YEAR = {2014},
    NUMBER = {5},
     PAGES = {957--1041}
}

@book {Liu,
    AUTHOR = {Liu, Qing},
     TITLE = {Algebraic geometry and arithmetic curves},
    SERIES = {Oxford Graduate Texts in Mathematics},
    VOLUME = {6},
      NOTE = {Translated from the French by Reinie Ern\'{e},
              Oxford Science Publications},
 PUBLISHER = {Oxford University Press, Oxford},
      YEAR = {2002}
}

@article {Patzt,
    AUTHOR = {Patzt, Peter},
     TITLE = {Representation stability for diagram algebras},
   JOURNAL = {J. Algebra},
  FJOURNAL = {Journal of Algebra},
    VOLUME = {638},
      YEAR = {2024},
     PAGES = {625--669}
}

@article {Moselle,
    AUTHOR = {Moselle, Isaac},
     TITLE = {Homological stability for {I}wahori-{H}ecke algebras of type {$B_n$}},
   JOURNAL = {J. Pure Appl. Algebra},
  FJOURNAL = {Journal of Pure and Applied Algebra},
    VOLUME = {228},
      YEAR = {2024},
    NUMBER = {5},
     PAGES = {Paper No. 107560, 26}
}

@article {Brauer,
    AUTHOR = {Brauer, Richard},
     TITLE = {On algebras which are connected with the semisimple continuous
              groups},
   JOURNAL = {Ann. of Math. (2)},
  FJOURNAL = {Annals of Mathematics. Second Series},
    VOLUME = {38},
      YEAR = {1937},
    NUMBER = {4},
     PAGES = {857--872}
}

@article {TemperleyLieb,
    AUTHOR = {Temperley, H. N. V. and Lieb, E. H.},
     TITLE = {Relations between the ``percolation'' and ``colouring''
              problem and other graph-theoretical problems associated with
              regular planar lattices: some exact results for the
              ``percolation'' problem},
   JOURNAL = {Proc. Roy. Soc. London Ser. A},
  FJOURNAL = {Proceedings of the Royal Society. London. Series A.
              Mathematical, Physical and Engineering Sciences},
    VOLUME = {322},
      YEAR = {1971},
    NUMBER = {1549},
     PAGES = {251--280}
}

@article{GrimmWarnaar,
title = {Solvable RSOS models based on the dilute BWM algebra},
journal = {Nuclear Physics B},
volume = {435},
number = {3},
pages = {482-504},
year = {1995},
author = {Uwe Grimm and S. Ole Warnaar}
}

@article {MartinMazorchuk,
    AUTHOR = {Martin, Paul and Mazorchuk, Volodymyr},
     TITLE = {On the representation theory of partial {B}rauer algebras},
   JOURNAL = {Q. J. Math.},
  FJOURNAL = {The Quarterly Journal of Mathematics},
    VOLUME = {65},
      YEAR = {2014},
    NUMBER = {1},
     PAGES = {225--247}
}

@misc{MazorchukPre,
    AUTHOR = {Mazorchuk, Volodymyr},
     TITLE = {On the Structure of Brauer Semigroup and its Partial Analogue},
    URL = {http://www2.math.uu.se/~mazor/PREPRINTS/CHIP/chip.pdf},
    YEAR = {1998},
    NOTE = {Available on the author's webpage}
}

@book {Benson,
    AUTHOR = {Benson, D. J.},
     TITLE = {Representations and cohomology. {I}},
    SERIES = {Cambridge Studies in Advanced Mathematics},
    VOLUME = {30},
   EDITION = {Second},
      NOTE = {Basic representation theory of finite groups and associative
              algebras},
 PUBLISHER = {Cambridge University Press, Cambridge},
      YEAR = {1998}
}

@article {Nakaoka,
    AUTHOR = {Nakaoka, Minoru},
     TITLE = {Decomposition theorem for homology groups of symmetric groups},
   JOURNAL = {Ann. of Math. (2)},
  FJOURNAL = {Annals of Mathematics. Second Series},
    VOLUME = {71},
      YEAR = {1960},
     PAGES = {16--42}
}

@article{FGG,
title = "Meander, folding, and arch statistics",
author = "{Di Francesco}, P. and O. Golinelli and E. Guitter",
year = "1997",
volume = "26",
pages = "97--147",
journal = "Mathematical and Computer Modelling",
publisher = "Elsevier Limited",
number = "8-10",
}

@article{FGG-TL,
title = "Meanders and the Temperley-Lieb algebra",
author = "{Di Francesco}, P. and O. Golinelli and E. Guitter",
year = "1997",
volume = "186",
pages = "1--59",
journal = "Commun. Math. Phys.",
publisher = "Elsevier Limited"
}

@article {Kauffman,
    AUTHOR = {Kauffman, Louis H.},
     TITLE = {An invariant of regular isotopy},
   JOURNAL = {Trans. Amer. Math. Soc.},
  FJOURNAL = {Transactions of the American Mathematical Society},
    VOLUME = {318},
      YEAR = {1990},
    NUMBER = {2},
     PAGES = {417--471}
}

@article {GrahamLehrer,
    AUTHOR = {Graham, J. J. and Lehrer, G. I.},
     TITLE = {Cellular algebras},
   JOURNAL = {Invent. Math.},
  FJOURNAL = {Inventiones Mathematicae},
    VOLUME = {123},
      YEAR = {1996},
    NUMBER = {1},
     PAGES = {1--34}
}

@article {KonigXi2,
    AUTHOR = {K\"{o}nig, Steffen and Xi, Changchang},
     TITLE = {Cellular algebras: inflations and {M}orita equivalences},
   JOURNAL = {J. London Math. Soc. (2)},
  FJOURNAL = {Journal of the London Mathematical Society. Second Series},
    VOLUME = {60},
      YEAR = {1999},
    NUMBER = {3},
     PAGES = {700--722}
}

@incollection {KonigXi1,
    AUTHOR = {K\"{o}nig, Steffen and Xi, Changchang},
     TITLE = {On the structure of cellular algebras},
 BOOKTITLE = {Algebras and modules, {II} ({G}eiranger, 1996)},
    SERIES = {CMS Conf. Proc.},
    VOLUME = {24},
     PAGES = {365--386},
 PUBLISHER = {Amer. Math. Soc., Providence, RI},
      YEAR = {1998}
}

@article {BHP2,
    AUTHOR = {Boyd, Rachael and Hepworth, Richard and Patzt, Peter},
     TITLE = {The homology of the partition algebras},
   JOURNAL = {Pacific J. Math.},
  FJOURNAL = {Pacific Journal of Mathematics},
    VOLUME = {327},
      YEAR = {2023},
    NUMBER = {1},
     PAGES = {1--27}
}

\end{document}